\documentclass[journal, 10pt]{IEEEtran}

\usepackage{amssymb,upgreek,amsmath}
\usepackage{makecell, boldline}
\usepackage{booktabs}
\usepackage{mathtools, bm}
\usepackage{esint}
\usepackage{subfigure, tabularx, amsthm}
\usepackage{enumitem, cite, color}
\usepackage{multirow}

\usepackage[OT2,T1]{fontenc}
\DeclareSymbolFont{cyrletters}{OT2}{wncyr}{m}{n}
\DeclareMathSymbol{\Sha}{\mathalpha}{cyrletters}{"58}

\newtheorem{prop}{Proposition}
\newtheorem{thm}{Theorem}
\newtheorem{lemma}{Lemma}

\newtheorem{definition}{Definition}

\DeclareMathOperator*{\argmin}{arg\,min}
\newcolumntype{Y}{>{\centering\arraybackslash}X}
\newcolumntype{M}[1]{>{\centering\arraybackslash}m{#1}}

\newcommand\Tstrut{\rule{0pt}{3.6ex}}

\newcommand\Bstrut{\rule[-2ex]{0pt}{0pt}}

\def\mytitle{Periodic Splines and Gaussian Processes for the Resolution of Linear Inverse Problems}

\begin{document}

\title{\mytitle}
\author{\thanks{Copyright (c) 2015 IEEE. Personal use of this material is
permitted. However, permission to use this material for any other purposes must
be obtained from the IEEE by sending a request to
pubs-permissions@ieee.org.}Ana\"{i}s Badoual$^{\ast}$, Julien Fageot$^{\ast}$, and Michael
Unser\thanks{The authors are with the
Biomedical Imaging Group, {\'{E}}cole polytechnique f{\'{e}}d{\'{e}}rale de Lausanne (EPFL),
Switzerland. This work was funded by the Swiss National Science
Foundation under Grant 200020-162343 and by the ERC grant agreement No 692726. 

$^\ast$ These authors contributed equally to this work.}}
\markboth{\mytitle}{Badoual {\textit{et al.}}: \mytitle}

\maketitle

\begin{abstract}
This paper deals with the resolution of inverse problems in a periodic setting or, in other terms, the reconstruction of periodic continuous-domain signals from their noisy measurements.
We focus on two reconstruction paradigms: variational and statistical.
In the variational approach, the reconstructed signal is solution to an optimization problem that establishes a tradeoff between fidelity to the data and smoothness conditions via a quadratic regularization associated to a linear operator.
In the statistical approach, the signal is modeled as a stationary random process defined from a Gaussian white noise and a whitening operator; one then looks for the optimal estimator in the mean-square sense.
We give a generic form of the reconstructed signals for both approaches,
allowing for a rigorous comparison of the two. 
We fully characterize the conditions under which the two formulations yield the same solution, which is a periodic spline in the case of sampling measurements. 
We also show that this equivalence between the two approaches remains valid on simulations for a broad class of problems.
This extends the practical range of applicability of the variational method.
\end{abstract}

\section{Introduction}\label{sec:intro}
This paper deals with inverse problems: one aims at recovering an unknown signal from its corrupted measurements.
To be more specific, the motivation of this work is the reconstruction of an unknown \emph{continuous-domain} and \emph{periodic} signal $f$ from its $M$ noisy measurements $y_m \approx \langle \nu_m, f \rangle = \int_0^1 \nu_m(t)f(t) \mathrm{d} t $ for $m=1\ldots M$, 
 where the $\nu_m$ are measurement functions.
The goal is then to build an output signal ${f}_{\mathrm{opt}}$ that is as close as possible to $f$. 

\subsection{Inverse Problems in the Continuous Domain} 

Inverse problems are often formulated in the discrete domain~\cite{Banham1997digital,Karayiannis1990regularization,Figueiredo2003algorithm,Afonso2011augmented,Bertero1998introduction}.	
This is motivated by the need of manipulating digital data on computers.
Nevertheless, many naturally occurring signals depend on continuous variables (\emph{e.g.}, time or position). This leads us to attempt recovering a signal $f_{\mathrm{opt}}(t)$ that depends on the continuous variable $t \in [0,1]$. In contrast with the classical discrete setting, our search space for this reconstructed signal is thus infinite-dimensional~\cite{Adcock2015generalized}. Moreover, we choose a regularization based on true derivatives (as opposed to finite differences) to impose some smoothness on the reconstructed signal, a concept that is absent in the discrete setting. 
	
When considering continuous-domain reconstruction methods, a majority of works, typically in machine learning, deal with sampling measurements. The goal is then to recover $f$ from its (possibly noisy) values $y_m\approx f(t_m)$ at fixed location $t_m$. 
In order to investigate a more general version of inverse problems, we shall consider generalized measurements~\cite{Papoulis1977generalized,Eldar2006minimum}. They largely exceed the sampling case and include Fourier sampling or convolution (\emph{e.g.}, MRI, x-ray tomography~\cite{Piccolomini2002regularization, Bostan2013Sparse}). 
Our only requirement is that the measurements $y_m$ depend linearly on, and evolve continuously with, the unknown signal $f$ up to some additive noise, so that $y_m \approx \langle \nu_m, f \rangle$.

\subsection{Variational vs. Statistical Methods}

In the discrete domain, two standard strategies are used to reconstruct 
an input signal $\bf{x}$ from its noisy measurements $\bf{y} \approx \mathbf{H} x$, where $\mathbf{H}$ models the acquisition process~\cite{Bertero1998introduction}.
The first approach is deterministic and can be tracked back to the '60s with Tikhonov's seminal work~\cite{Tikhonov1963solution}.
The ill-posedness of the problem usually imposes the addition of a regularizer.
By contrast, Wiener filtering is based on the stochastic modelization of the signals of interest and the optimal estimation of the targeted signal $\bf{x}$.
This paper generalizes these ideas for the reconstruction of \emph{continuous} signals from their \emph{discrete} measurements.

In the variational setting, the reconstructed signal is a solution to an optimization problem that imposes some smoothness conditions~\cite{cassel2013variational}. More precisely, the optimization problem may take the form 
\begin{equation}\label{eq: intro}
{f}_{\mathrm{opt}} = \argmin_f \bigg( \sum_{m=1}^M  \big(y_m -\langle \nu_m, f \rangle \big)^2 + \lambda \| \mathrm{L}f\| ^2_{L_2} \bigg),
\end{equation}
where $\mathrm{L}$ is a linear operator.
The first term in~\eqref{eq: intro} controls the data fidelity. The regularization term $\| \mathrm{L}f\| ^2_{L_2} $ constrains the function to satisfy certain smoothness properties (for this reason, the variational approach is sometimes called a smoothing approach).
The parameter $\lambda$ in~\eqref{eq: intro} quantifies the tradeoff between the fidelity to the data and the regularization constraint.

In the statistical setting, the signal is modeled as a random process and is optimally reconstructed using estimation theory~\cite{Moon2000mathematical}. 
More precisely, one assumes that the continuous-domain signal is the realization of a stochastic process $s$ and that the samples are given by $y_m= \langle \nu_m , s \rangle + \epsilon_m$, where $\epsilon_m$ is a random perturbation and $\nu_m$ a linear measurement function.
In this case, one specifies the reconstructed signal as the optimal statistical estimator in the mean-square sense
\begin{equation}
	{f}_{\mathrm{opt}}  = \argmin_{\tilde{s}} \mathbb{E} \left[ \lVert s - \tilde{s}(\cdot | \mathbf{y}) \rVert_{L_2}^2 \right],
\end{equation}
where the estimators $t \mapsto \tilde{s}(t | \mathbf{y})$ are computed from the generalized samples $y_m$. The solution depends on the measurement function $\nu_m$ and the stochastic models specified for $s$ and $\epsilon_m$. In our case, the random process $s$ is characterized by a linear operator $\mathrm{L}$ that is assumed to have a whitening effect (it transforms $s$ into a periodic Gaussian white noise), while the perturbation is i.i.d. Gaussian.

\subsection{Periodic and General Setting}
The variational and statistical approaches have been extensively studied for continuous-domain signals defined on the infinitely supported real line. 
However, it is often assumed in practice that the input signals are periodic. 
In fact, a standard computational approach to signal processing is to extend by periodization the signals of otherwise bounded support. 
Periodic signals arise also naturally in applications such as the parametric representation of closed curves~\cite{Cohen1994part, Delgado2012ellipse,Badoual2017subdivision}. 
This has motivated the development of signal-processing tools and techniques specialized to periodic signals in sampling theory, error analysis, wavelets, stochastic modelization, or curve representation~\cite{Vetterli2002FRI,Maravic2005sampling,blu2008sparse,Jacob2002sampling,Triebel2008function, Fageot2017besov,Badoual2016inner}. 

In this paper, we develop the theory of the variational and statistical approaches for periodic continuous-domain signals in a very general context, including the following aspects:
\begin{itemize}[leftmargin=10 pt]
\item We consider a broad class of measurement functions, with the only assumptions that they are linear and continuous. 

\item Both methods refer to an underlying linear operator $\mathrm{L}$ that affects the smoothness properties of the reconstruction. We deal with a very broad class of linear operators acting on periodic functions. 

\item We consider possibly non-quadratic data fidelity terms in the smoothing approach.
\end{itemize}

	\subsection{Related Works}
The topics investigated in this paper have already received some attention in the literature, mostly in the non-periodic setting.
	
	\paragraph{Reconstruction over the Real Line}
Optimization problems of the form~\eqref{eq: intro} appear in many fields and receive different names, including inverse problems in image processing~\cite{Bertero1998introduction}, representer theorems in machine learning~\cite{Scholkopf2001generalized}, or sometimes interpolation elsewhere.
Schoenberg was the first to show the connection between~\eqref{eq: intro} and spline theory~\cite{Schoenberg1964spline}. 
Since then, this has been extended to other operators~\cite{Unser2007self}, or to the interpolation of the derivative of the signal~\cite{Uhlmann2016hermite,Condat2011quantitative}.
Many recent methods are dealing with non-quadratic regularization, especially the ones interested in the reconstruction of sparse discrete~\cite{Candes2006sparse,Donoho2006} or continuous signals~\cite{Adcock2015generalized,Denoyelle2015support,Unser2016splines,gupta2018continuous}. 
We discuss this aspect more extensively in Section~\ref{sec:TV}.

A statistical framework requires the specification of the noise and of the signal stochastic model. The signal is then estimated from its measurements. A classical measure of the quality of an estimator is the mean-square error. This criterion is minimized by the minimum mean-square error (MMSE) estimator~\cite{Moon2000mathematical,Tarantola2005inverse}.
The theory has been developed mostly for Gaussian processes and in the context of sampling measurements~\cite{Berlinet2011reproducing}. We are especially interested in innovation models, for which one assumes that the signal can be whitened (\emph{i.e.}, transformed into a white noise) by the application of a linear operator~\cite{Kailath1968innovationsA,Unser2014unifiedContinuous}. Non-periodic models have been studied in many situations, including the random processes associated with differential~\cite{Kimeldorf1970spline, Uhlmann2015SampTA} or fractional operators~\cite{Blu2007self}. Extensions to non-Gaussian models are extensively studied by Unser and Tafti~\cite{Unser2014sparse}. 
	
The statistical and variational frameworks are deeply connected. 
It is remarkable that the solution of either problem can be expressed as spline functions in relation with the linear operator $\mathrm{L}$ involved in regularization (variational approach) or whitening (statistical approach).
Wahba has shown that the two approaches are strictly equivalent in the case of stationary Gaussian models~\cite{Wahba1990spline}. This equivalence has also been recognized by several authors since then, as shown by Berlinet and Thomas-Agnan~\cite{Berlinet2011reproducing}, and Unser and Blu~\cite{Unser2005generalized}.
In the non-stationary case, this equivalence is not valid any more and the existence of connections has received less attention. 

	\paragraph{Reconstruction of Periodic Signals}
Some strong practical concerns have motivated the need for an adaptation of the theory to the periodic setting.
Important contributions in that direction have been proposed.
Periodic splines are constructed and applied to sampling problems by Schoenberg~\cite{Schoenberg1964trigonometric} and Golomb~\cite{Golomb1968approximation}. The smoothing spline approach is studied in the periodic setting by Wahba~\cite{Wahba1990spline} for derivative operators of any order. Although the periodic extension of the classical theory is briefly mentioned by several authors~\cite{Berlinet2011reproducing,Wahba1990spline,de1978practical}, we are not aware of a global treatment. Providing a general analysis in the periodic setting is precisely what we propose in this paper. 
	
	\subsection{Outline and Main Contributions}
Section~\ref{sec:mathematical_background} contains the main notations and tools for periodic functions and operators.
In Section~\ref{sec:representer_th}, we state the periodic representer theorem (Theorem~\ref{th: periodic representer thm}). It fully specifies the form of the solution in the variational approach in a very general setting. For the specific case of sampling measurements, we show that this solution is a periodic spline (Proposition~\ref{prop: Spline}).
Section~\ref{sec:processes_MMSE} is dedicated to the statistical approach. We introduce a class of periodic stationary processes (the \emph{Gaussian bridges}) for which we specify the MMSE estimator in the case of generalized linear measurements (Theorem~\ref{th: MMSE}). We also provide a theoretical comparison between the variational and statistical approaches by reformulating the MMSE estimation as the solution of a new optimization problem (Proposition~\ref{prop: opt pb}). This highlights the strict equivalence of the two approaches for invertible operators and extends known results from sampling to generalized linear measurements.
For non-invertible operators, we complete our analysis with simulations in Section~\ref{sec:simulations}. In particular, we give empirical evidence of the practical relevance of the variational approach for the reconstruction of periodic stationary signals.
We provide in Section~\ref{sec:periodicvsline} a comparison between our results in the periodic setting and the known results over the real line.
Finally, we conclude in Section~\ref{sec:conslusions}.
All the proofs have been postponed to the Appendix sections.

\section{Mathematical Background for Periodic Signals}\label{sec:mathematical_background}

Throughout the paper, we consider periodic functions and random processes. Without loss of generality, the period can always be normalized to one. Moreover, we identify a periodic function over $\mathbb{R}$ with its restriction to a single period, chosen to be $\mathbb{T} = [0,1)$.
We use the symbols $f$, $s$, and $\tilde{s}$ to specify a function, a random process, and an estimator of $s$, respectively. 

We call $\mathcal{S}(\mathbb{T})$ the space of $1$-periodic functions that are infinitely differentiable, $\mathcal{S}^{\prime}(\mathbb{T})$ the space of $1$-periodic generalized functions (dual of $\mathcal{S}(\mathbb{T})$), and $L_2(\mathbb{T})$ the Hilbert space of square integrable $1$-periodic functions associated with the norm ${\lVert f \rVert_{L_2} = (\int_0^1 \lvert f(t) \rvert^2 \mathrm{d} t )^{1/2}}$. Working with $\mathcal{S}'(\mathbb{T})$ allows us to deal with functions with no pointwise interpretation, such as the Dirac comb defined by
\begin{equation}
\Sha=\sum\limits_{k \in \mathbb{Z}}\delta(\cdot -k),
\end{equation} 
where $\delta$ is the Dirac impulse. 
The duality product between an element $f \in \mathcal{S}'(\mathbb{T})$ and a smooth function $g \in \mathcal{S}(\mathbb{T})$ is denoted by $\langle f, g \rangle$. For instance, $\langle \Sha , g \rangle = g(0)$ for every $g$.
When the two real functions are in $L_2(\mathbb{T})$, we simply have the usual scalar product $\langle f , g \rangle = \int_{0}^1 f(t) {g(t)} \mathrm{d} t$.
All these concepts are extended to complex-valued functions in the usual manner with the convention that $\langle f , g \rangle = \int_{0}^1 f(t) \overline{g(t)} \mathrm{d}t$ for square-integrable functions.
The complex sinusoids are denoted by $e_k(t)=\mathrm{e}^{\mathrm{j} 2\pi k t}$ for any $k \in \mathbb{Z}$ and $t \in \mathbb{T}$. 
Any periodic generalized function $f \in \mathcal{S}'(\mathbb{T})$ can be expanded as
\begin{equation}\label{eq: Fourier serie}
f(t)=\sum_{k\in \mathbb{Z}}\widehat{f}[k] \mathrm{e}^{\mathrm{j} 2\pi k t}=\sum_{k\in \mathbb{Z}}\widehat{f}[k] e_k(t),
\end{equation}
where the $\widehat{f}[k]$ are the Fourier coefficients of $f$, given by
$\widehat{f}[k] = \langle f , e_k \rangle$. Finally, the convolution between two periodic functions $f$ and $g$ is given by
\begin{equation}\label{eq: periodic convolution}
(f \ast g)(t)=\langle f, g(t-\cdot)\rangle.
\end{equation}
If $f, g \in L_2(\mathbb{T})$, we have that $(f \ast g)(t)=\int_0^1f(\tau)g(t-\tau)\mathrm{d}\tau$.

\subsection{Linear and Shift-Invariant Operators} \label{subsec:LSIop}
Let $\mathrm{L}$ be a linear, shift-invariant (LSI), and continuous operator from $\mathcal{S}(\mathbb{T})$ to $\mathcal{S}'(\mathbb{T})$.
The shift invariance implies the existence of $\widehat{L}[k] \in \mathbb{C}$ such that
\begin{equation}\label{eq: LSI}
\mathrm{L}e_k=\widehat{L}[k]e_k,
\end{equation}
for any $k \in \mathbb{Z}$. We call $\widehat{L}[k]$ the frequency response of the operator $\mathrm{L}$; it is also given by
\begin{equation}
\widehat{L}[k]= \langle \mathrm{L} \{\Sha\}, e_k \rangle = \int_0^1 \mathrm{L}\{\Sha\}(t)\mathrm{e}^{-\mathrm{j} 2\pi k t} \mathrm{d}t.
\end{equation}

The sequence $(\widehat{L}[k])$ is the Fourier series of the periodic generalized function $\mathrm{L}\{\Sha\}$, and is therefore of slow growth~\cite[Chapter VII]{Schwartz1966distributions}.
This implies that $\mathrm{L}$, a priori from $\mathcal{S}(\mathbb{T})$ to $\mathcal{S}'(\mathbb{T})$, actually continuously maps $\mathcal{S}(\mathbb{T})$ into itself. This is a significant difference with the non-periodic setting --- we discuss this point in the conclusion in Section~\ref{sec:conslusions}. Therefore, one can extend it by duality from $\mathcal{S}^{\prime}(\mathbb{T})$ to $\mathcal{S}^{\prime}(\mathbb{T})$. Then, for every $f \in \mathcal{S}^{\prime}(\mathbb{T})$, we easily obtain from~\eqref{eq: LSI} that
\begin{equation}
\mathrm{L}f(t)=\sum_{k\in \mathbb{Z}}\widehat{(\mathrm{L}f)}[k]e_k(t), \ \mbox{where } \widehat{(\mathrm{L}f)}[k]=\widehat{f}[k]\widehat{L}[k].
\end{equation}

The null space of $\mathrm{L}$ is $\mathcal{N}_{\mathrm{L}} = \{ f \in \mathcal{S}'(\mathbb{T}) \ | \ \mathrm{L} f = 0\}$.
We shall only consider operators whose null space is finite-dimensional, in which case $\mathcal{N}_{\mathrm{L}}$ can only be made of linear combinations of sinusoids at frequencies that are annihilated by $\mathrm{L}$. 
We state this fact in Proposition~\ref{prop: null space} and prove it in Appendix~\ref{app: null space}.

\begin{prop}\label{prop: null space}
Let $\mathrm{L}$ be a continuous LSI operator. If $\mathrm{L}$ has a finite-dimensional null space $\mathcal{N}_{\mathrm{L}}$ of dimension $N_0$, then the null space is of the form
\begin{equation}\label{eq: null space}
\mathcal{N}_{\mathrm{L}}=\mathrm{span}\{e_{k_n}\}_{n=1}^{N_0}, 
\end{equation}
where the $k_n \in \mathbb{Z}$ are distinct.
\end{prop}

\smallskip
\noindent From~\eqref{eq: LSI} and~\eqref{eq: null space}, we deduce that $\widehat{L}[k]=0$ if and only if $k=k_n$ for some $n\in [1 \ldots N_0]$.
In the following, we consider real-valued operators. In that case, we have the Hermitian symmetry
$ \overline{\widehat{L}[-k]}={\widehat{L}}[k]$. Moreover, $e_{k_n} \in \mathcal{N}_{\mathrm{L}}$ if and only if $e_{-k_n} \in \mathcal{N}_{\mathrm{L}}$.
The orthogonal projection of $f$ on the null space $\mathcal{N}_{\mathrm{L}}$ is given by
\begin{equation}\label{eq: Proj null space}
\mathrm{Proj}_{\mathcal{N}_{\mathrm{L}}} \{f\}=\sum_{n=1}^{N_0}\widehat{f}[k_n]e_{k_n}.
\end{equation}

Let ${\mathcal{K}_{\mathrm{L}}=\mathbb{Z}\backslash \{k_n\}_{ n \in \{1 \ldots N_0\}}}$. Then,~\eqref{eq: Fourier serie} can be re-expressed as
${f=\mathrm{Proj}_{\mathcal{N}_{\mathrm{L}}} \{f\}+ \sum_{k \in \mathcal{K}_{\mathrm{L}}} \widehat{f}[k]e_k}$ and we have that
$	\mathrm{L} f (t)= \sum_{k \in \mathcal{K}_{\mathrm{L}}} \widehat{f}[k] \widehat{L}[k] e_k(t),
$ which yields the Parseval relation
\begin{equation}\label{eq: Parseval}
\int_0^1| \mathrm{L}f (t)|^2 \mathrm{d}t=\sum_{k \in \mathcal{K}_{\mathrm{L}}} \big | \widehat{f}[k] \big |^2 \big |\widehat{L}[k] \big |^2 .
\end{equation}

\subsection{Periodic $\mathrm{L}$-Splines} \label{sec : periodic splines}
Historically, splines are functions defined to be piecewise polynomials~\cite{Schoenberg1973cardinal}. A spline is hence naturally associated to the derivative operator of a given order~\cite{Unser1999splines} in the sense that, for a fixed $N\geq 1$, a spline function $f : \mathbb{R}\rightarrow \mathbb{R}$ satisfies ${\mathrm{L} f (t) = \sum a_m \delta(t-t_m)}$ with $\mathrm{L} = \mathrm{D}^N$ the $N$th derivative.
Splines have been extended to differential~\cite{schumaker2007spline,Schultz1967Lsplines,Unser2005cardinal,Unser2005think}, fractional~\cite{Unser2007self,Panda2006fractional} or, more generally, spline-admissible operators~\cite{Unser2014sparse}.
We adapt here this notion to the periodic setting, where the Dirac impulse $\delta$ is replaced by the Dirac comb $\Sha$. 

\begin{definition} \label{def:Lspline}
Consider an LSI operator $\mathrm{L}$ with finite-dimensional null space.
We say that a function $f$ is a \emph{periodic $\mathrm{L}$-spline} if
\begin{equation}
\mathrm{L} f (t) = \sum_{m=1}^M a_m \Sha( t - t_m)
\end{equation}
for some integer $M \geq 1$, weights $a_m \in \mathbb{R}$, and knot locations $t_m \in \mathbb{T}$. 
\end{definition}

Periodic $\mathrm{L}$-splines play a crucial role in the variational and statistical approaches for the resolution of inverse problems in the periodic setting. We represent some periodic splines associated to different operators in Figure~\ref{fig: spline}. 
 \begin{figure}
 \centering
 \hspace*{-0.2cm}
 \subfigure[$\mathrm{L}=\mathrm{D}+\mathrm{I}$ and $M=4$]{\includegraphics[scale=0.32]{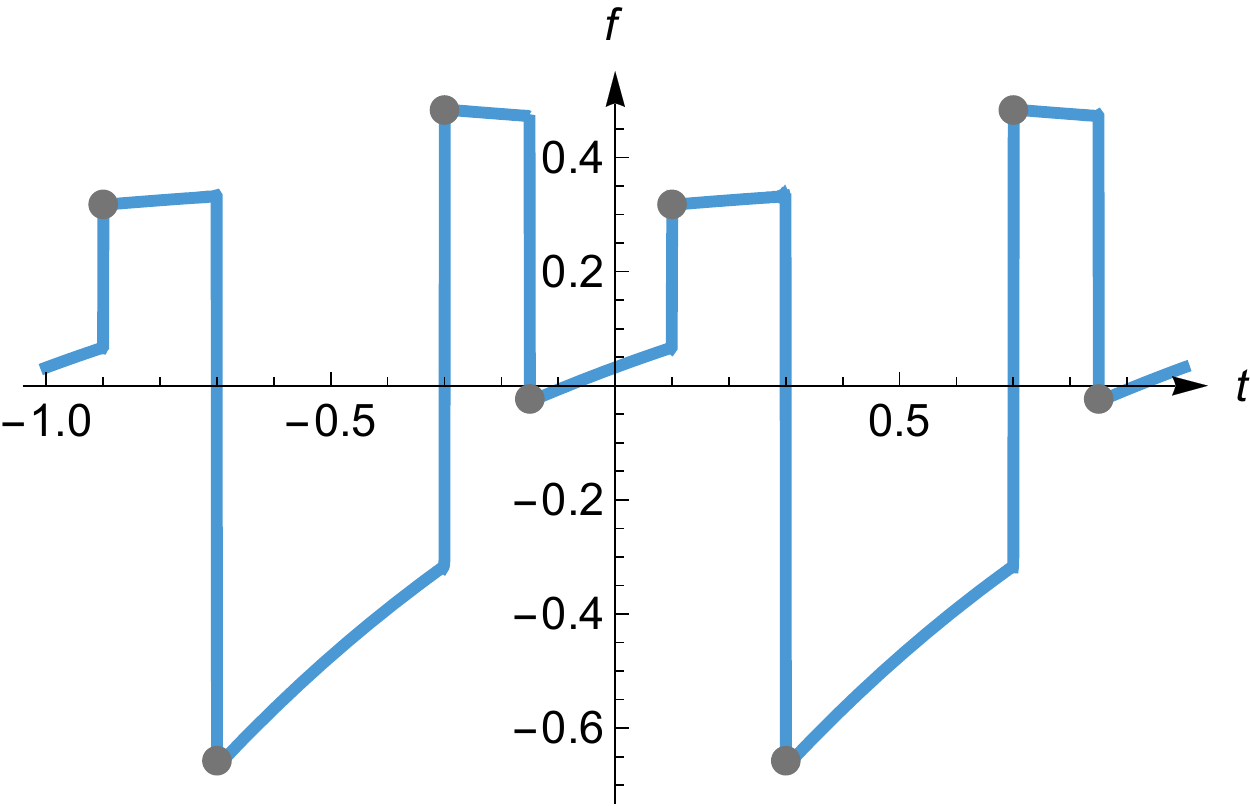}}
  \hspace*{+0.2cm}
  \subfigure[$\mathrm{L}=\mathrm{D}^2$ and $M=5$]{\includegraphics[scale=0.32]{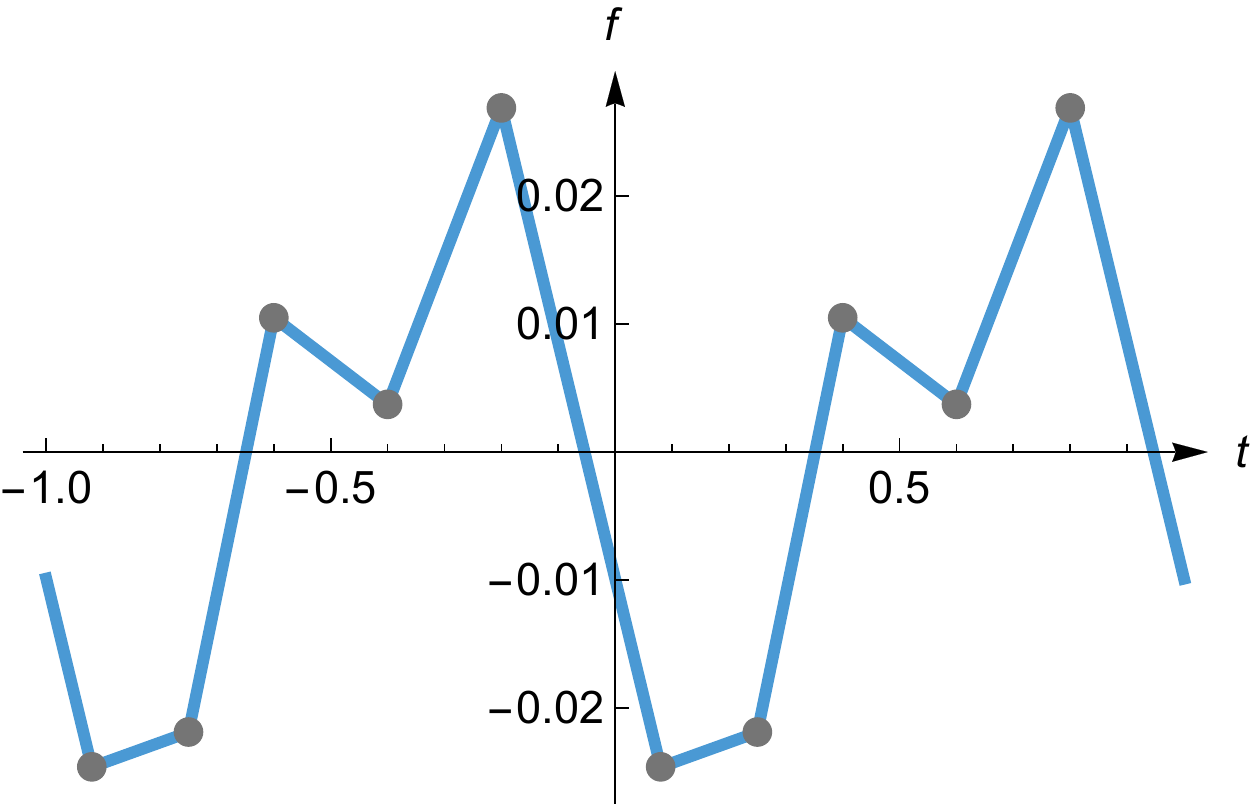}}
 \caption{Illustrations of periodic $\mathrm{L}$-splines. Dots: nodes $\big( t_m, f(t_m) \big)$. The spline in (a) corresponds to the periodization of an exponential B-spline (see Figure 1 in~\cite{Unser2005cardinal}). }
 \label{fig: spline}
 \end{figure}

\section{Periodic Representer Theorem}\label{sec:representer_th}
We now consider a continuous LSI operator $\mathrm{L}$ with finite-dimensional null space $\mathcal{N}_{\mathrm{L}}$. Let $\boldsymbol{\nu}$ be the vector of the linear measurement functions $\nu_1, \ldots, \nu_M$.
They usually are of the form $\nu_m = \delta(\cdot - t_m)$ for time-domain sampling problems. 
Here, we consider general linear measurements to include any kind of inverse problems.
 In this section, our goal is to recover a function $f$ from observed data $\mathbf{y}=(y_1, \ldots ,y_M)$ such that $y_m \simeq \langle \nu_m, f\rangle$.
To do so, we consider the variational problem
\begin{equation}\label{eq: opt pb}
\underset{f}{\min}\bigg( F(\mathbf{y},\boldsymbol\nu (f))+\lambda \| \mathrm{L}f\| _{L_2}^2\bigg),
\end{equation}
where $F: \mathbb{R}^M \times \mathbb{R}^M \rightarrow \mathbb{R}^+$ is a strictly convex and continuous function called the \textit{cost function}. This function controls the fidelity to data. A special attention will be given to the quadratic data fidelity of the form
\begin{equation}\label{eq: quadratic F}
F(\mathbf{y},\boldsymbol\nu (f))=\sum\limits_{m=1}^M (y_m-\langle \nu_m, f\rangle )^2.
\end{equation}
We give the solution of~\eqref{eq: opt pb} for the space of $1$-periodic functions in Theorem~\ref{th: periodic representer thm}. To derive this solution, we first introduce and characterize the space of functions on which~\eqref{eq: opt pb} is well-defined.

\subsection{Search Space} \label{sec:searchspace}
The optimization problem~\eqref{eq: opt pb} deals with functions such that $\mathrm{L} f$ is square-integrable, which leads us to introduce ${\mathcal{H}_{\mathrm{L}}=\{f \in \mathcal{S}^{\prime} (\mathbb{T}) \ | \  \mathrm{L}f \in L_2(\mathbb{T})\}}$. Due to~\eqref{eq: Parseval}, we have that
\begin{equation}\label{eq: search space}
\mathcal{H}_{\mathrm{L}}=\{f \in \mathcal{S}^{\prime} (\mathbb{T}) \ | \  \sum_{k\in \mathcal{K}_{\mathrm{L}}}|\widehat{f}[k]|^2|\widehat{L}[k]|^2 <  + \infty\}.
\end{equation}
Similar constructions have been developed for functions over $\mathbb{R}$ or for sequences by Unser \textit{et al.}~\cite{Unser2016splines,Unser2016representer}. 
We now identify a natural Hilbertian structure on $\mathcal{H}_{\mathrm{L}}$. If $\mathrm{L}: \mathcal{H}_{\mathrm{L}} \rightarrow L_2(\mathbb{T})$ is invertible, then $\mathcal{H}_{\mathrm{L}}$ inherits the Hilbert-space structure of $L_2$ via the norm $\|\mathrm{L}f\|_{L_2}$. However, when $\mathrm{L}$ has a nontrivial null space, $\|\mathrm{L}f \|_{L_2}$ is only a semi-norm, in which case there exists $f \neq 0$ (any element of the null space of $\mathrm{L}$) such that $\| \mathrm{L}f \|_{L_2}=0$. To obtain a \textit{bona fide} norm, we complete the semi-norm with a special treatment for the null-space components in Proposition~\ref{prop: inner product}.

\begin{prop}\label{prop: inner product}
Let $\mathrm{L}$ be a continuous LSI operator whose finite-dimensional null space is defined by ${\mathcal{N}_{\mathrm{L}}=\mathrm{span}\{e_{k_n}\}_{n=1}^{N_0}}$.
We fix $\gamma^2 > 0$. Then, $\mathcal{H}_{\mathrm{L}}$ is a Hilbert space for the inner product
\begin{equation}\label{eq: inner product}
\langle f,g\rangle _{\mathcal{H}_{\mathrm{L}}}=\langle \mathrm{L} f, \mathrm{L} g\rangle  +\gamma^2 \sum_{n=1}^{N_0}\widehat{f}[k_n]\overline{\widehat{g}[k_n]}.
\end{equation}
\end{prop}
\noindent The proof is given in Appendix~\ref{app: inner-product}. We have that
${\| f \|_{\mathcal{H}_{\mathrm{L}}}^2=\| \mathrm{L}f\|_{L_2}^2+ \gamma^2 \| \mathrm{Proj}_{\mathcal{N}_{\mathrm{L}}} \{f\}\|_{L_2}^2}$, where $\mathrm{Proj}_{\mathcal{N}_{\mathrm{L}}} \{f\}$ is given by~\eqref{eq: Proj null space}. The coefficient $\gamma^2$ balances the contribution of both terms.

\subsection{Periodic Reproducing-Kernel Hilbert Space} \label{sec:conditionRKHS}
Reproducing-kernel Hilbert spaces (RKHS) are Hilbert spaces on which the evaluation maps $f \mapsto f(t)$ are well-defined, linear, and continuous.
In this section, we answer the question of when the Hilbert space $\mathcal{H}_{\mathrm{L}}$ associated to an LSI operator $\mathrm{L}$ with finite-dimensional null space is a RKHS. This property is relevant to us because periodic function spaces that are RKHS are precisely the ones for which one can use measurement functions of the form $\nu_m = \Sha(\cdot - t_m)$ in~\eqref{eq: opt pb}.

\begin{definition}\label{def: RKHS}
Let $\mathcal{H}\subseteq \mathcal{S}^{\prime} (\mathbb{T})$ be a Hilbert space of $1$-periodic functions and $\mathcal{H}^{\prime}$ be its dual. Then, we say that $\mathcal{H}$ is a \emph{RKHS} if the shifted Dirac comb $\Sha (\cdot - t_0) \in \mathcal{H}^{\prime}$ for any $t_0 \in \mathbb{T}$. 
\end{definition}

This implies that any element $f$ of a RKHS has a pointwise interpretation as a function $t \rightarrow f(t)$.
As is well known, for any RKHS there exists a unique function $h : \mathbb{T} \times \mathbb{T} \rightarrow \mathbb{R}$ such that $h(\cdot, t_0) \in \mathcal{H}^{\prime}$ and ${\langle f, h(\cdot, t_0)\rangle=f(t_0),}$ for every $ t_0 \in \mathbb{T}$ and $ f \in \mathcal{H}$. We call $h$ the \emph{reproducing kernel} of $\mathcal{H}$.

\begin{prop}\label{prop: RKHS}
Let $\mathrm{L}$ be a continuous LSI operator with finite-dimensional null space. The Hilbert space $\mathcal{H}_{\mathrm{L}}$ (see~\eqref{eq: search space}) is a RKHS if and only if 
\begin{equation} \label{eq:condition RKHS}
\sum\limits_{k \in \mathcal{K}_{\mathrm{L}}}\frac{1}{|\widehat{L}[k]|^2}< + \infty.
\end{equation}
Then, the reproducing kernel for the scalar product~\eqref{eq: inner product} is given by $h(t,\tau)=h_{\gamma}(t-\tau)$, where $h_{\gamma} \in \mathcal{S}'(\mathbb{T})$ is
\begin{equation}\label{eq: reproducing kernel}
h_{\gamma}(t)=\sum_{n=1}^{N_0}\frac{e_{k_n}(t)}{\gamma^2}+\sum_{k \in \mathcal{K}_{\mathrm{L}}}\frac{e_k(t)}{|\widehat{L}[k]|^2}.
\end{equation}
\end{prop}

\smallskip
\noindent The proof is given in Appendix~\ref{app: RKHS}. Note that the reproducing kernel only depends on the difference $(t-\tau)$.

\subsection{Periodic Representer Theorem}
Now that we have defined the search space of the optimization problem~\eqref{eq: opt pb}, we derive the representer theorem that gives the explicit form of its unique periodic solution.

\begin{thm}\label{th: periodic representer thm}
We consider the optimization problem 
\begin{equation}\label{eq: periodic representer thm}
\underset{f \in \mathcal{H}_{\mathrm{L}}}{\min}\bigg( F(\mathbf{y},\boldsymbol\nu (f))+\lambda \| \mathrm{L}f\| _{L_2}^2\bigg),
\end{equation}
where
\begin{itemize}
	\item $F: \mathbb{R}^M \times \mathbb{R}^M \rightarrow \mathbb{R}^+$ is strictly convex and continuous;
	\item $\mathrm{L}$ is an LSI operator with finite-dimensional null space;
	\item $\boldsymbol \nu = (\nu_1, \dots ,\nu_M) \in (\mathcal{H}'_{\mathrm{L}})^M$ such that $\mathcal{N}_{\mathrm{L}}\cap \mathcal{N}_{\boldsymbol \nu}=\{0\}$;
	\item $\mathbf{y}=(y_1, \ldots ,y_M) \in \mathbb{R}^M$ are the observed data; and 
	\item $\lambda > 0$ is a tuning parameter.
\end{itemize}

\noindent Then,~\eqref{eq: periodic representer thm} 
 has a unique solution of the form 
\begin{equation}\label{eq:formsolutionRT}
f_{\mathrm{RT}}(t)=\sum\limits_{m=1}^M a_m \varphi_m(t)+\sum_{n=1}^{N_0}b_n e_{k_n}(t),
\end{equation}
where $a_m, b_n \in \mathbb{R}$, $\varphi_m=h_{\gamma} \ast \nu_m$, and $h_{\gamma}$ is given by~\eqref{eq: reproducing kernel}. Moreover, the vector $\mathbf{a} = (a_1, \ldots  ,a_{M})$ satisfies the relation $\mathbf{P}^{\mathsf{T}} \mathbf{a} = \mathbf{0}$, with $\mathbf{P}$ the $(M\times N_0)$ matrix with entries ${[\mathbf{P}]_{m,n} = \langle e_{k_n}, \nu_m \rangle}$.
\end{thm}

\noindent The proof of Theorem~\ref{th: periodic representer thm} is given in Appendix~\ref{app: per rep th}. 
The optimal solution depends on $(M+N_0)$ coefficients, but the condition $\mathbf{P}^\mathsf{T} \mathbf{a} = \mathbf{0}$ implies that there are only $(M+N_0-N_0)= M$ degrees of freedom. 
In the case when $F$ is quadratic of the form~\eqref{eq: quadratic F}, the solution is made explicit in Proposition~\ref{prop: interpolation pb}.

\begin{prop}\label{prop: interpolation pb}
Under the conditions of Theorem~\ref{th: periodic representer thm}, if $F$ is given by~\eqref{eq: quadratic F}, then the vectors $\mathbf{a}$ and $\mathbf{b}$ satisfy the linear system
\begin{equation} \label{eq:matrixform}
\begin{pmatrix}
\mathbf{a} \\
\mathbf{b}
\end{pmatrix}=
\begin{pmatrix}
\mathbf{G}+ \lambda \mathbf{I} & \mathbf{P} \\
\mathbf{P}^{\mathsf{T}} & \mathbf{0}
\end{pmatrix}^{-1}
\begin{pmatrix}
\mathbf{y} \\
\mathbf{0}
\end{pmatrix}, 
\end{equation}
where $\mathbf{P} \in \mathbb{C}^{M\times N_0}$ is defined by $[\mathbf{P}]_{m,n}=\langle e_{k_n}, \nu_m \rangle$ and $\mathbf{G} \in \mathbb{R}^{M \times M}$ is a Gram matrix such that 
\begin{equation}
\label{eq:Chgamma}
[\mathbf{G}]_{m_1,m_2} = \int_0^1\int_0^1 \nu_{m_1}(t)h_{\gamma}(t-\tau)\nu_{m_2}(\tau)\mathrm{d}t \mathrm{d}\tau.\end{equation}
\end{prop}

\noindent The proof is given in Appendix~\ref{app: interpolation pb}. 
{In the case of sampling measurements, we show moreover in Proposition~\ref{prop: Spline} that the optimal solution is a periodic spline in the sense of Definition~\ref{def:Lspline}.
We recall that such measurements are valid as soon as the search space $\mathcal{H}_{\mathrm{L}}$ is a RKHS, a situation that has been fully characterized in Proposition~\ref{prop: RKHS}.}

\begin{prop}\label{prop: Spline}
{Under the conditions of Proposition~\ref{prop: interpolation pb}, if $\mathrm{L}$ satisfies~\eqref{eq:condition RKHS} and if the measurements are of the form ${\nu_m=\Sha(\cdot-t_m)}$, $t_m \in \mathbb{T}$, then the unique solution of~\eqref{eq: periodic representer thm} is a periodic $(\mathrm{L}^*\mathrm{L})$-spline with weights $a_m$ and knots $t_m$. }
\end{prop}

\noindent The proof is given in Appendix~\ref{app: Spline}.

\section{Periodic Processes and MMSE}\label{sec:processes_MMSE}
In this section, we change perspective and consider the following statistical problem: given noisy measurements of a zero-mean and real periodic Gaussian process, we are looking for the optimal estimator (for the mean-square error) of the complete process over $\mathbb{T}$.

\subsection{Non-Periodic Setting} \label{sec:Non-Per Setting}
In a non-periodic setting, it is usual to consider stochastic models where the random process $s$ is a solution to the stochastic differential equation~\cite{Unser2014sparse}
\begin{equation} \label{eq:Ls=w}
\mathrm{L} s = w,
\end{equation}
where $\mathrm{L}$ is a linear differential operator and $w$ a continuous domain (non-periodic) Gaussian white noise. When the null space of the operator is nontrivial, it is necessary to add boundary conditions such that the law of the process $s$ is uniquely defined.

\subsection{Gaussian Bridges} \label{sec:GGB}
In the periodic setting, the construction of periodic Gaussian processes has to be adapted. We first introduce the notion of periodic Gaussian white noise, exploiting the fact that the law of a zero-mean periodic Gaussian process $s$ is fully characterized by its covariance function $r_s(t,\tau)$ such that
\begin{equation}\label{eq: esp Gaussian process}
\mathbb{E} [\langle s , f \rangle \langle s , g \rangle ] = \int_0^1 \int_0^1 f(t) r_s(t,\tau) \overline{g(\tau)} \mathrm{d} t \mathrm{d} \tau.
\end{equation}

\begin{definition}\label{def: Gaussian white noise}
A \emph{periodic Gaussian white noise}\footnote{Without loss of generality, we only consider Gaussian white noise with zero-mean and variance $1$.} is a Gaussian random process $w$ whose covariance is ${r_w(t,\tau) = \Sha(t-\tau)}$.
\end{definition}

For any periodic real function $f$, the random variable $\langle w, f \rangle$ is therefore Gaussian with mean $0$ and variance $\lVert f \rVert_{L_2}^2$. Moreover, $\langle w, f \rangle$ and $\langle w, g \rangle$ are independent if and only if $\langle f , g \rangle = 0$. Hence, the Fourier coefficients $\widehat{w}[k]=\langle w,\mathrm{e}_k \rangle$ of the periodic Gaussian white noise satisfy the following properties:
\begin{itemize}
\item $\widehat{w}[k]=\Re(\widehat{w}[k])+ \mathrm{j} \ \Im(\widehat{w}[k])$;
\item $\overline{\widehat{w}[-k]}=\widehat{w}[k]$;
\item $\Re(\widehat{w}[k]), \ \Im(\widehat{w}[k]) \sim \mathcal{N}(0, \frac{1}{2})$, $\forall k>0$;
\item $\widehat{w}[0] \in \mathbb{R}$ and $\widehat{w}[0] \sim \mathcal{N}(0, 1)$;
\item $\Re(\widehat{w}[k]), \ \Im(\widehat{w}[k])$, and $\widehat{w}[0]$ are independent.
\end{itemize}

\smallskip
\noindent Put differently, for any nonzero frequency $k$, $\mathbb{E}[\widehat{w}[k]^2]=0$ and ${\mathbb{E}[\widehat{w}[k] \overline{\widehat{w}[k]}]=1}$. This means that $\widehat{w}[k]$, $k\neq 0$, follows a complex normal distribution with mean 0, covariance $1$, and pseudo-covariance $0$~\cite{Goodman1963}.
\begin{table*}
\caption{Gaussian bridges for several operators.}\label{Table: Gaussian Bridges}
\vspace*{-0.3cm}
\begin{center}
\begin{tabularx}{\textwidth}{M{1.2cm} | Y  | Y  | Y  | Y }
\hline
\hline
\multicolumn{1}{M{1cm}  |}{} & $\mathrm{D}+ \mathrm{I}$ & $\mathrm{D}$ & $\mathrm{D}^2+ 4 \pi^2 \mathrm{I}$ & $\mathrm{D}^2$\Tstrut\Bstrut\\
\hline
 $\widehat{L}[k]$ & $\mathrm{j}2\pi k +1$ & $\mathrm{j}2\pi k$ & $4\pi ^2(1-k^2)$  & $-4\pi^2k^2$\Tstrut\\
$\mathcal{N}_{\mathrm{L}}$ & $\mathrm{span}\{0\}$ & $\mathrm{span}\{e_0\}$ & $\mathrm{span}\{e_1, e_{-1}\}$ & $\mathrm{span}\{e_0\}$\Tstrut\Bstrut\\
Gaussian bridges \tiny{$\gamma_0^2=1$} &
	\begin{minipage}{0.5\textwidth}
 		 \includegraphics[width=0.4\columnwidth]{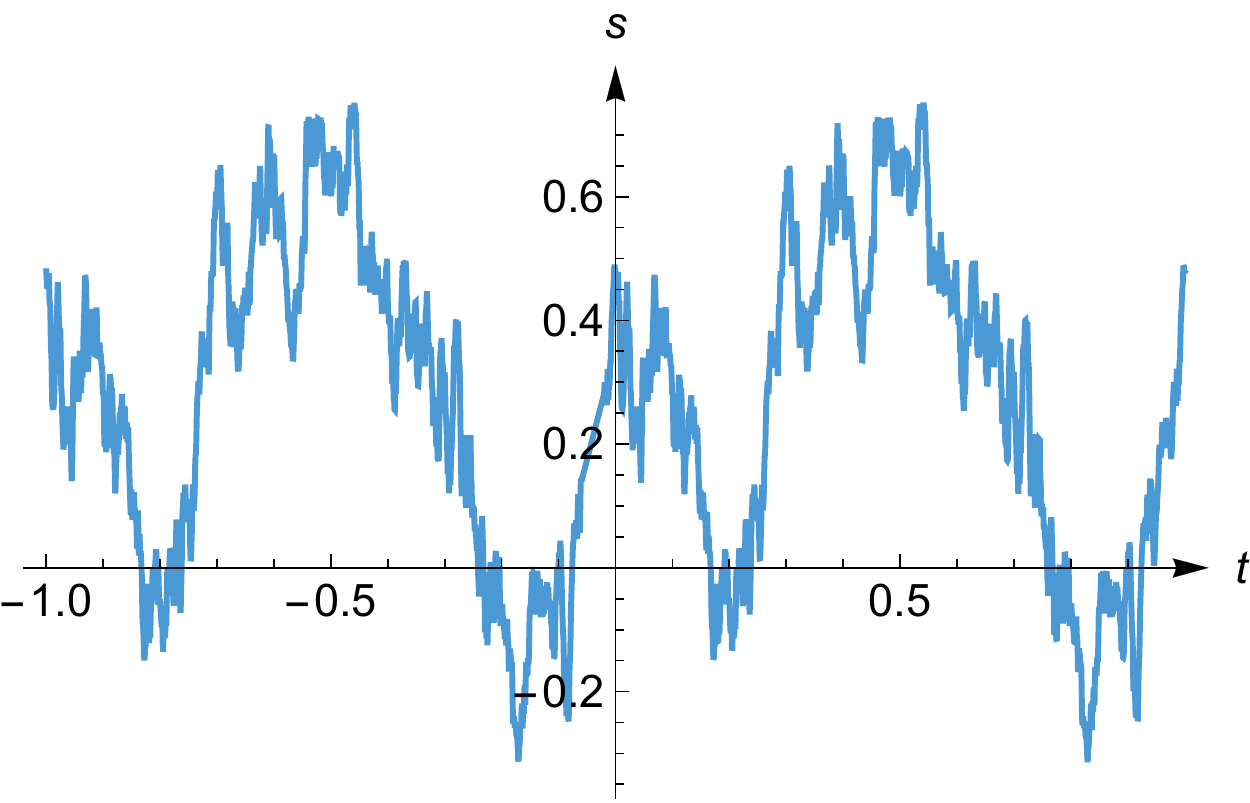} 
  	\end{minipage}& 
  	 \begin{minipage}{0.5\textwidth}
   		 \includegraphics[width=0.4\columnwidth]{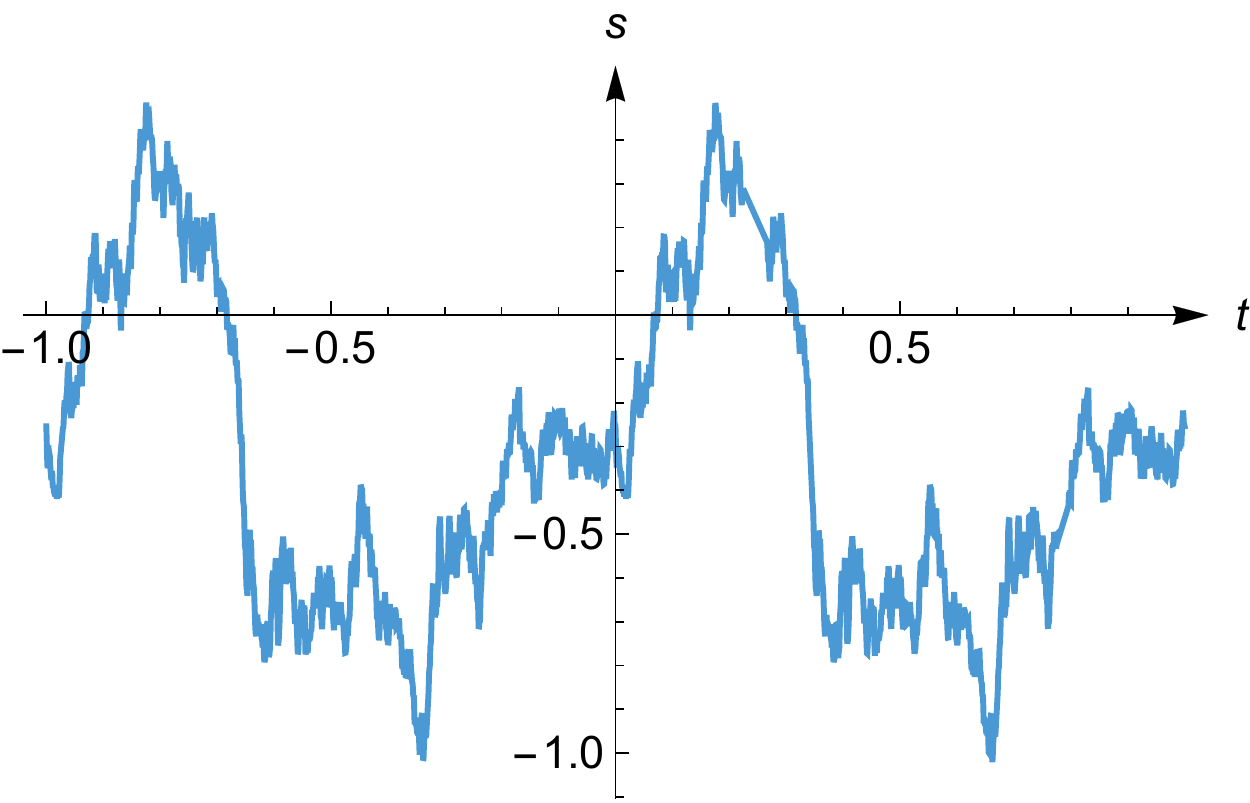}
 	\end{minipage}& 
  	 \begin{minipage}{0.5\textwidth}
   	    \includegraphics[width=0.4\columnwidth]{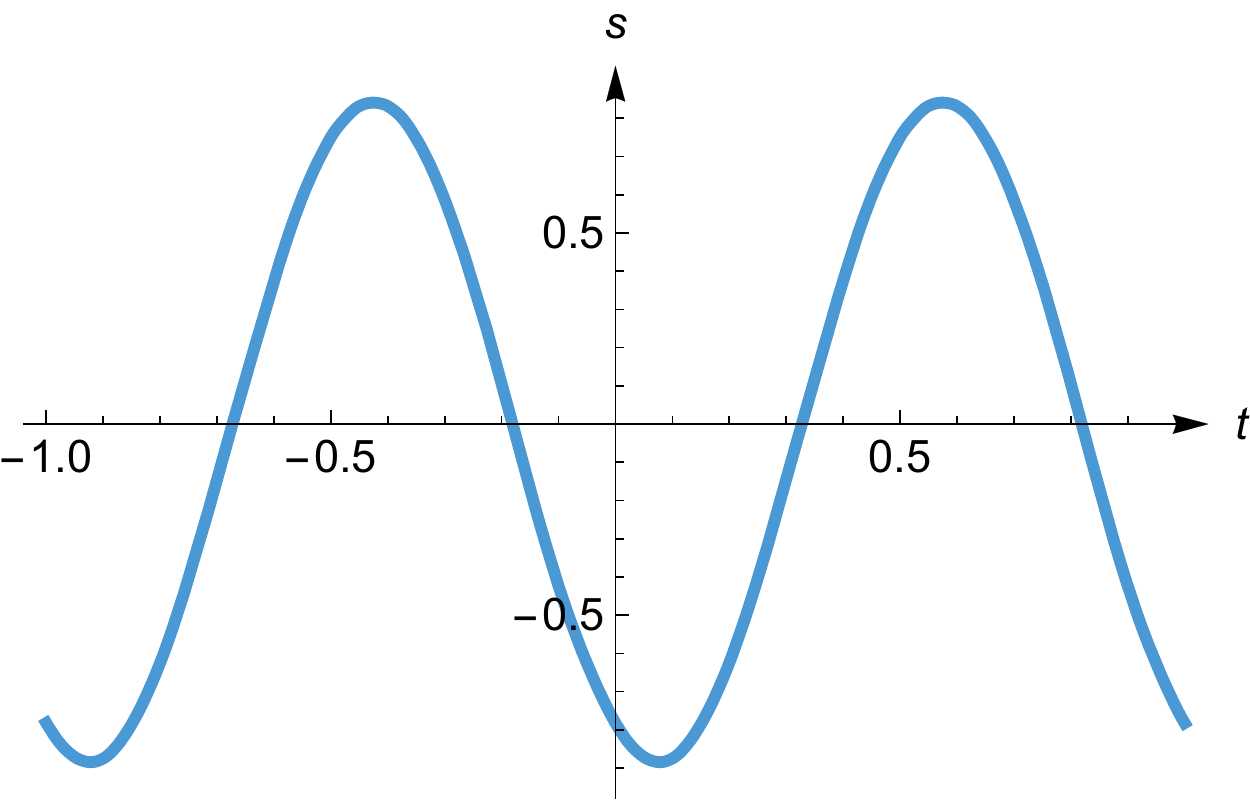}
   \end{minipage}&
   \begin{minipage}{0.5\textwidth}
   	\includegraphics[width=0.4\columnwidth]{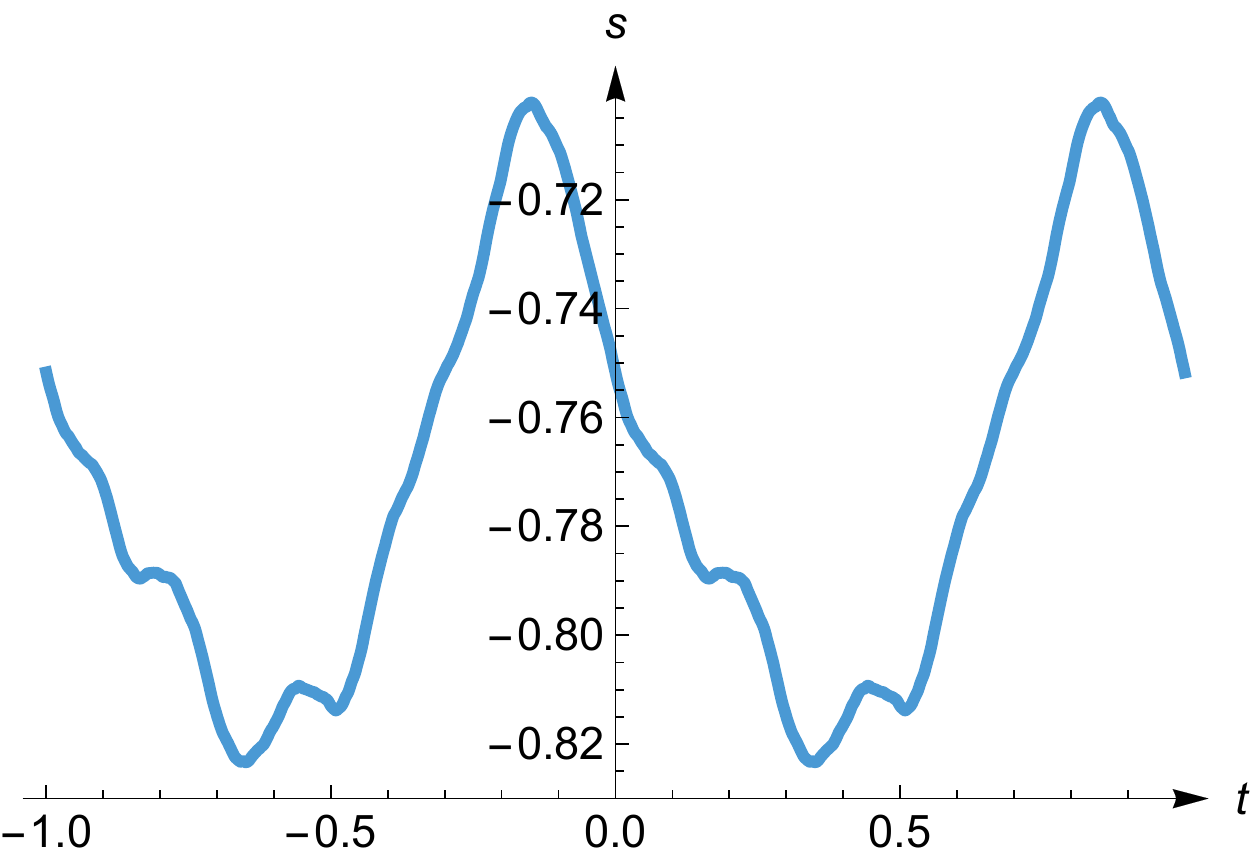}
   \end{minipage}\Tstrut\Bstrut\\
\hline
\hline
\end{tabularx}
\end{center}\vspace*{-0.2cm}
\end{table*}

When $\mathrm{L}$ has a nontrivial null space, there is no hope to construct a periodic process $s$ solution of~\eqref{eq:Ls=w} with $w$ a periodic Gaussian white noise. Indeed, the operator $\mathrm{L}$ kills the null-space frequencies, which contradicts that $\widehat{w}[k_n] \neq 0$ almost surely for $n=1\ldots N_0$. One should adapt~\eqref{eq:Ls=w} accordingly by giving special treatment to the null-space frequencies. 
We propose here to consider a new class of periodic Gaussian processes: the \emph{Gaussian bridges}. Given some operator $\mathrm{L}$ and $\gamma_0 >0$, we set
\begin{equation}\label{eq:Lgamma}
\mathrm{L}_{\gamma_0} = \mathrm{L} + \gamma_0 \mathrm{Proj}_{\mathcal{N}_{\mathrm{L}}},
\end{equation}
where $\mathrm{Proj}_{\mathcal{N}_{\mathrm{L}}}$ is given by~\eqref{eq: Proj null space}. Note that $\mathrm{L}_{\gamma_0} = \mathrm{L}$ for any $\gamma_0$ when the null space of $\mathrm{L}$ is trivial.
Moreover, we remark that
\begin{equation}
	\lVert \mathrm{L}_{\gamma_0} f \rVert_{L_2}^2 = 	\lVert \mathrm{L}  f \rVert_{L_2}^2 + 	\gamma_0^2 \lVert \mathrm{Proj}_{\mathcal{N}_{\mathrm{L}}}  \{f\} \rVert_{L_2}^2 = 	\lVert  f \rVert_{\mathcal{H}_{\mathrm{L}}}^2,
\end{equation}
where $\lVert f \rVert_{\mathcal{H}_{\mathrm{L}}}^2 = \langle f, f \rangle_{\mathcal{H}_{\mathrm{L}}}$ is given in~\eqref{eq: inner product} (with $\gamma =\gamma_0$).

\begin{definition}\label{def: model process}
A \emph{Gaussian bridge} is a periodic Gaussian process $s$, solution to the stochastic differential equation
\begin{equation} \label{eq:GB}
	\mathrm{L}_{\gamma_0} s = w,
\end{equation}
with $w$ a periodic Gaussian white noise and $\mathrm{L}_{\gamma_0}$ given by~\eqref{eq:Lgamma} for some LSI operator $L$ with finite-dimensional null space and ${\gamma_0} > 0$. We summarize this situation with the notation $s\sim \mathcal{GB}(\mathrm{L}, \gamma_0^2)$. When the null space is trivial, in which case the parameter $\gamma_0^2$ is immaterial, we write $s \sim\mathcal{GB}(\mathrm{L})$. 
\end{definition}

\noindent The Gaussian-bridge terminology is inspired by the Brownian bridge, the periodic version of the Brownian motion\footnote{Our definition differs from the classical one, in which the Brownian bridge is zero at the origin instead of being zero-mean~\cite{Revuz2013continuous}.}.
Several realizations of our Gaussian bridges for various operators are shown in Table~\ref{Table: Gaussian Bridges} for $\gamma_0^2=1$. 
The influence of the parameter $\gamma_0^2$ is illustrated in Figure~\ref{fig: influence of gamma0}. 

\begin{prop} \label{prop:covariancebridge}
	The covariance function of the Gaussian bridge $s \sim \mathcal{GB}(\mathrm{L}, \gamma_0^2)$ is
	\begin{equation}\label{eq: covariance GB}
		r_s(t,\tau) = h_{\gamma_0} (t-\tau),
	\end{equation}
	where $h_{\gamma_0}$ is defined in~\eqref{eq: reproducing kernel}. It implies that
\begin{equation}\label{eq: esperance GB}
\mathbb{E} [\langle s , f \rangle \langle s , g \rangle ] =\langle h_{\gamma_0} \ast f,g\rangle.
\end{equation}
In particular, we have that
\begin{equation} \label{eq:variancesk}
\mathbb{E} [\lvert \widehat{s} [k]\rvert^2 ] = \widehat{h}_{\gamma_0} [k].
\end{equation}
\end{prop}

\noindent The proof of Proposition~\ref{prop:covariancebridge} is given in Appendix~\ref{app:covariance}. An important consequence is that a Gaussian bridge is stationary since its covariance function only depends on the difference $(t-\tau)$. 

\subsection{Measurement Model and MMSE Estimator}

For this section, we restrict ourselves to operators $\mathrm{L}$ for which the native space $\mathcal{H}_{\mathrm{L}}$ is a RKHS.
In that case, using~\eqref{eq:variancesk} and~\eqref{eq: reproducing kernel}, the Gaussian bridge $s$ satisfies
\begin{equation}
\mathbb{E} [ \lVert s \rVert_{L_2}^2]
= \sum_{k\in  \mathbb{Z}} \mathbb{E} [\lvert \widehat{s} [k]\rvert^2 ] = \sum_{k \in \mathcal{K}_{\mathrm{L}}} \frac{1}{\lvert \widehat{L} [k] \rvert^2} + \sum_{n=1}^{N_0} \frac{1}{\gamma_0^2},
\end{equation}
which is finite according to~\eqref{eq:condition RKHS}. Therefore, the Gaussian bridge $s$ is (almost surely) square-integrable.

The observed data $\mathbf{y}$ are assumed to be generated as
\begin{equation}\label{eq: noisy measurement}
	\mathbf{y}= \langle \boldsymbol\nu , s \rangle + \boldsymbol{\epsilon},
\end{equation}
where $s \sim \mathcal{GB}(\mathrm{L}, \gamma_0^2)$ is a Gaussian bridge (see Definition~\ref{def: model process}), $\boldsymbol{\nu} = ( \nu_1, \ldots , \nu_M) $ is a vector of $M$ linear measurement functions, and $\boldsymbol{\epsilon}$ are independent random perturbations such that $\boldsymbol{\epsilon}\sim \mathcal{N}(\mathbf{0}, \sigma_0^2 \mathbf{I})$.
Given $\mathbf{y}$ in~\eqref{eq: noisy measurement}, we want to find the estimator $\tilde{s}$ of the Gaussian bridge $s$, imposing that it minimizes the quantity $\mathbb{E} [ \lVert s - \tilde{s} \rVert_2^2]$. 

\begin{thm}\label{th: MMSE}
Let $\mathbf{y}=(y_1, \ldots ,y_M)$ be the noisy measurement vector~\eqref{eq: noisy measurement} of the Gaussian bridge $s \sim \mathcal{GB}(\mathrm{L}, \gamma_0^2)$, with measurement functions $\nu_m \in \mathcal{H}_{\mathrm{L}}'$, $m=1\ldots M$. Then, 
the MMSE estimator of $s$ given the samples $\{y_m\}_{m \in [1 \ldots M]}$ is 
\begin{equation}
\tilde{s}_{\mathrm{MMSE}}(t)=\sum_{m=1}^M d_m \varphi_m(t),
\end{equation}
\noindent where $\varphi_m=h_{\gamma_0} \ast \nu_m$ with $\nu_m \in \mathcal{H}_{\mathrm{L}}'$, $\mathbf{d}=(d_1, \ldots ,d_M)=(\mathbf{G}+ \sigma_0^2 \mathbf{I})^{-1}\mathbf{y}$, and $\mathbf{G}$ is the Gram matrix defined in~\eqref{eq:matrixform}.
\end{thm}

\noindent The proof is given in Appendix~\ref{app: MMSE}. 
{Theorem~\ref{th: MMSE} can be seen as a generalization of the classical Wiener filtering, designed for discrete signals, to the hybrid case where the input signal is in a (periodic) continuous-domain and the (finite-dimensional) measurements are discrete.}
A leading theme of this paper is that the form of the MMSE estimator $\tilde{s}_{\mathrm{MMSE}}$ is very close to the one of the solution of the representer theorem $f_{\mathrm{RT}}$ with $\lambda = \sigma_0^2$ and for a quadratic cost function. This connection is exploited in Section~\ref{sec:discussion}. 

\subsection{MMSE Estimation as a Representer Theorem}\label{sec:discussion}
The MMSE estimator given in Theorem~\ref{th: MMSE} can be interpreted as the solution of the optimization problem described in Proposition~\ref{prop: opt pb}.

\begin{prop}\label{prop: opt pb}
Consider an LSI operator $\mathrm{L}$ with finite-dimensional null space, $\gamma >0$, and $\nu_m \in \mathcal{H}_{\mathrm{L}}'$ for ${m=1\ldots M}$. We set $\mathrm{L}_{\gamma}$ as in~\eqref{eq:Lgamma}. 
Then, the solution of the optimization problem
\begin{equation} \label{eq:L2reguWithGamma}
\underset{f \in \mathcal{H}_{\mathrm{L}}}{\min} \bigg ( \sum\limits_{m=1}^M (y_m-\langle f,\nu_m \rangle )^2 +\lambda   \| \mathrm{L}_\gamma f \|_{L_2}^2   \bigg )
\end{equation}
exists, is unique, and given by
\begin{equation} \label{eq:optigamma} 
f_\mathrm{opt}(t)= \sum\limits_{m=1}^M d_m \varphi_m(t),
\end{equation}
where $\varphi_m=h_\gamma \ast \nu_m$ and $\mathbf{d}=(d_1, \ldots ,d_M)=(\mathbf{G}+ \lambda \mathbf{I})^{-1}\mathbf{y}$.
 In particular, the unique minimizer of~\eqref{eq:L2reguWithGamma} is the MMSE estimator given in Theorem~\ref{th: MMSE} for $\lambda = \sigma^2_0$ and $\gamma = \gamma_0$.
\end{prop}

\smallskip
\noindent The proof of Proposition~\ref{prop: opt pb} follows the same steps as the ones of Theorem~\ref{th: periodic representer thm} (form of the minimizer for the periodic representer theorem) and Proposition~\ref{prop: interpolation pb} (explicit formulas in terms of system matrix for the vectors $\mathbf{a}$ and $\mathbf{b}$), with significant simplifications that are detailed in Appendix~\ref{sec:proofRTgamma}. Proposition~\ref{prop: opt pb} has obvious similarities with Theorem~\ref{th: periodic representer thm}, but it also adds new elements.

\smallskip
\begin{itemize}[leftmargin=*]
\item Proposition~\ref{prop: opt pb} gives an interpretation of the MMSE estimator of a Gaussian bridge given its measurements as the solution to an optimization problem. This problem is very close to the periodic representer theorem (Theorem~\ref{th: periodic representer thm}) for a quadratic cost function. However,~\eqref{eq:L2reguWithGamma} differs from~\eqref{eq: periodic representer thm} because the regularization also penalizes null-space frequencies.

\smallskip
\item If the null space $\mathcal{N}_{\mathrm{L}}$ is trivial, then 
\begin{equation}\label{eq:equivalence}
f_{\mathrm{RT}} = \tilde{s}_{\mathrm{MMSE}}
\end{equation}
 for ${\lambda=\sigma_0^2}$. This means that Theorem~\ref{th: periodic representer thm} (smoothing approach) and~\ref{th: MMSE} (statistical approach) correspond to the same reconstruction method. This equivalence is well-known for stationary processes on $\mathbb{R}$ in the case of time-domain sampling measurements~\cite{Wahba1990spline}. 
Our results extend this to the periodic setting and to the case of generalized linear measurements.

\smallskip
\item If the null space is nontrivial, then Theorem~\ref{th: periodic representer thm} and Proposition~\ref{prop: opt pb} yield different reconstructions. In particular, this implies that one cannot interpret the optimizer $f_{\mathrm{RT}}$ in Theorem~\ref{th: periodic representer thm} as the MMSE estimator of a Gaussian bridge. Yet, the solutions get closer and closer as $\gamma_0 \rightarrow 0$. In Section~\ref{sec:simulations}, we investigate more deeply this situation.
\end{itemize}

\section{Quality of the Estimators on Simulations}\label{sec:simulations}
We consider
$\tilde{s}_{\gamma, \lambda}(t|\mathbf{y})=\sum_{m=1}^M d_m \varphi_m(t)$ as the linear estimator of $s$ given $\mathbf{y}$, where $\varphi_m=h_{\gamma} \ast \nu_m$, $\mathbf{d}=(\mathbf{G}+ \lambda \mathbf{I})^{-1}\mathbf{y}$, and $\mathbf{G}$ is defined in Proposition~\ref{prop: interpolation pb}. 
To simplify notations, we shall omit $\mathbf{y}$ when considering $\tilde{s}_{\gamma,\lambda}( \cdot | \mathbf{y}) = \tilde{s}_{\gamma,\lambda}$.
Each pair $(\lambda,\gamma)$ gives an estimator. In particular, if $s$ is a Gaussian bridge, then ${\tilde{s}_{\mathrm{MMSE}}=\tilde{s}_{\gamma_0, \sigma_0^2}}$, according to Theorem~\ref{th: MMSE}. The mean-square error (MSE) of $\tilde{s}_{\gamma, \lambda}$ over $N$ experiments is computed as $\mbox{MSE}=\frac{1}{N}\sum_{n=1}^N \| s_n -\big(\tilde{s}_{\gamma,\lambda}\big)_n \|_{L_2}^2$, where the $s_n$ are independent realizations of $s$ that yield a new noisy measurement $\mathbf{y}_n$ and $\big(\tilde{s}_{\gamma, \lambda}\big)_n = \tilde{s}_{\gamma,\lambda}(\cdot |\mathbf{y}_n)$ is the estimator based on $\mathbf{y}_n$. We define the normalized mean-square error (NMSE) by
 \begin{equation}\label{eq: error lambda}
\mbox{NMSE}=\frac{\mbox{MSE}}{\frac{1}{N}\sum_{n=1}^N \| s_n \|_{L_2}^2}\approx \frac{\mathbb{E}[\| s-\tilde{s}_{\gamma,\lambda} \|_{L_2}^2]}{\mathbb{E}[\| s \|_{L_2}^2]}.
 \end{equation}

In this section, we first detail the generation of Gaussian bridges (Section~\ref{sec: generation gaussian bridge}). We then investigate the role of the parameters $\lambda$ (Section~\ref{sec: lambda}) and $\gamma^2$ (Section~\ref{sec: gamma}) on the quality of the estimator $\tilde{s}_{\gamma, \lambda}$. We primarily focus on time-domain sampling measurements with $\langle \boldsymbol{\nu} , s \rangle =(s(t_1), \ldots, s(t_M) )^{\mathsf{T}}$, where the $t_m$ are in $\mathbb{T}$.

\subsection{Generation of Gaussian Bridges}\label{sec: generation gaussian bridge}
We first fix the operator $\mathrm{L}$ with null space $\mathcal{N}_{\mathrm{L}}$ of dimension $N_0$ and $\gamma_0>0$. Then, we generate $(2N_{\mathrm{coef}}+1)$ Fourier coefficients $\{\widehat{w}[k]\}_{k\in [-N_{\mathrm{coef}} \ldots N_{\mathrm{coef}}]}$ of a Gaussian white noise according to Definition~\ref{def: Gaussian white noise}. Finally, we compute the Gaussian bridge $s$ as
\begin{align}\label{eq: Gaussian bridge}
s(t)
&=\sum\limits_{\substack{k \in \mathcal{K}_{\mathrm{L}} \\ |k|\leq N_{\mathrm{coef}}}} \frac{\widehat{w}[k]}{\widehat{L}[k] }e_k(t) + \sum\limits_{n=1}^{N_0} \frac{\widehat{w}[k_n]}{\gamma_0}e_{k_n}(t).
\end{align}
Since $N_0 < \infty$,~\eqref{eq: Gaussian bridge} provides a mere approximation of the Gaussian bridge. However, the approximation error can be made arbitrarily small by taking $N_{\mathrm{coef}}$ large enough.
In Figure~\ref{fig: influence of gamma0}, we generate $s \sim \mathcal{GB}(\mathrm{D}^2 + 4\pi^2 \mathrm{I}, \gamma_0^2 )$ for four values of $\gamma_0^2$.
For small values of $\gamma_0^2$, the null-space component dominates, which corresponds in this case to the frequency $|k|= 1$. When $\gamma_0^2$ increases, the null-space component has a weaker influence. 
 \begin{figure}
 \center
 	\subfigure[$\gamma_0^2=10^0$.]{
 		\includegraphics[width=0.4\columnwidth]{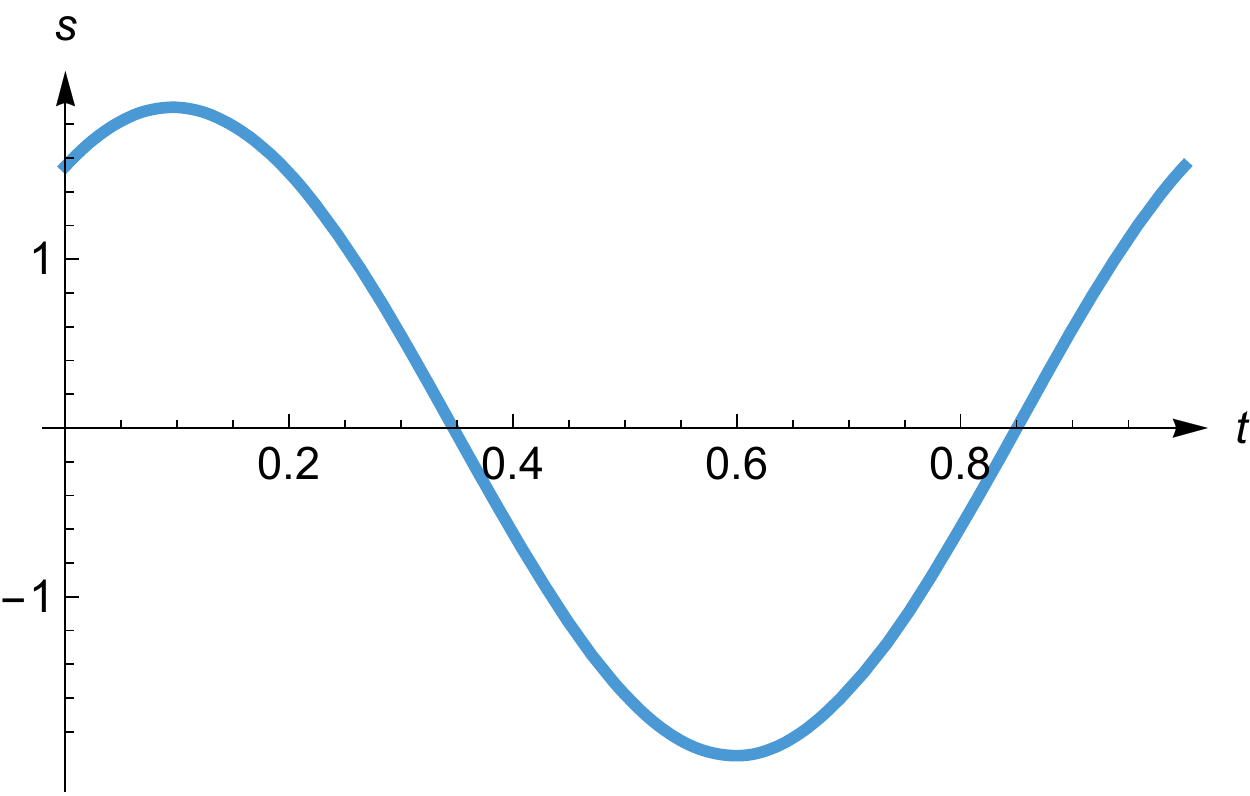}
 		}
 	\subfigure[$\gamma_0^2=10^2$.]{
 		\includegraphics[width=0.4\columnwidth]{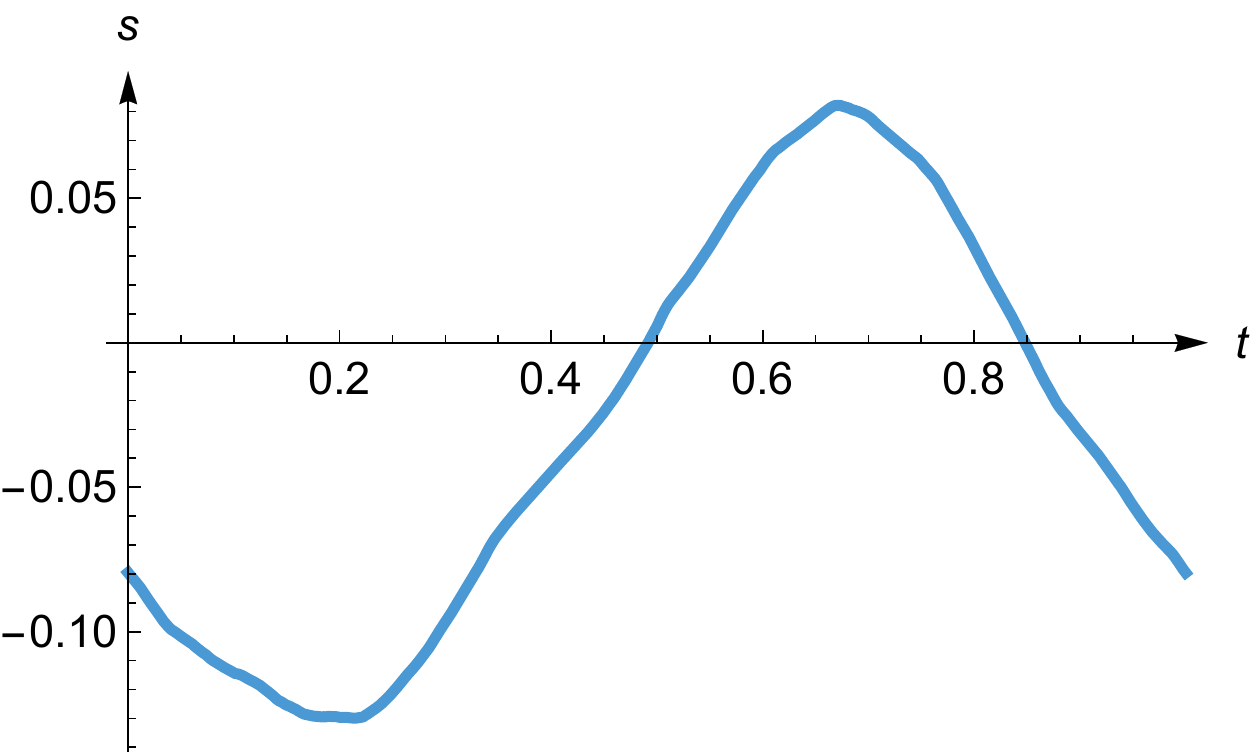}
 		}
 	\subfigure[$\gamma_0^2=10^3$.]{
 		\includegraphics[width=0.4\columnwidth]{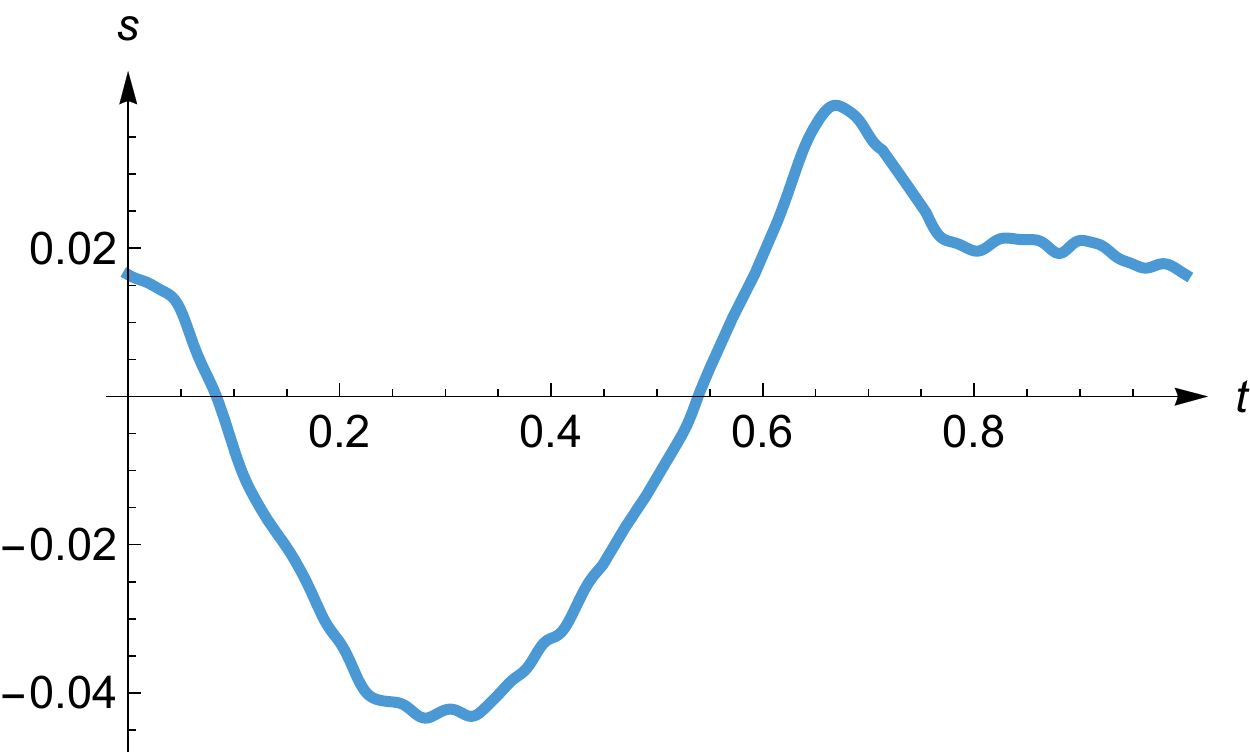}
 		}
 	\subfigure[$\gamma_0^2=10^6$.]{
 		\includegraphics[width=0.4\columnwidth]{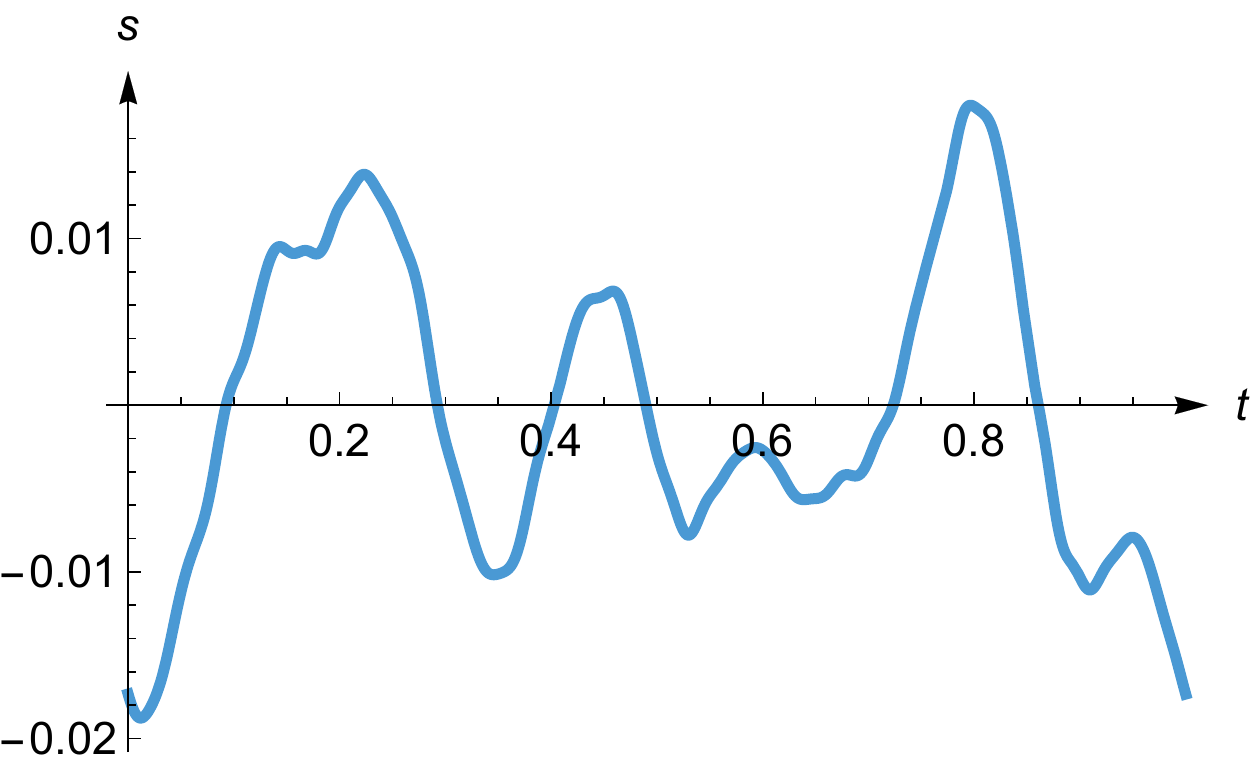}
 		}
 	\caption{Illustration of $s \sim \mathcal{GB}(\mathrm{D}^2+4\pi^2 \mathrm{I}, \gamma_0^2)$ for different values of $\gamma_0^2$.}
 	\label{fig: influence of gamma0}
 \end{figure}

 \subsection{Influence of $\lambda$}\label{sec: lambda}
 We evaluate the influence of the parameter $\lambda$ for the case of the invertible operator $\mathrm{L}= \mathrm{D} + \mathrm{I}$.
In this case we have that $\mathrm{Proj}_{\mathcal{N}_\mathrm{L}} = 0$ (since $\mathcal{N}_{\mathrm{L}} = \{0\}$), which simplifies~\eqref{eq:Lgamma}. Hence, the parameter $\gamma_0^2$ is immaterial and we denote by $\tilde{s}_\lambda$ the estimator associated to $\lambda > 0$. 
We consider $s \sim \mathcal{GB}(\mathrm{D} + \mathrm{I})$ and $\sigma_0^2 = 10^{-2}$.

\smallskip
\textit{Time-Domain Sampling Measurements.} 
We generated ${N=500}$ realizations of $s$. From each one, we extracted ${M=30}$ noisy measurements.
We then computed $30$ estimators $\{\big(\tilde{s}_{\lambda}\big)_n\}_{\lambda \in \mathcal{L}_1}$, where $\mathcal{L}_1$ is the set of values obtained by uniform sampling of the interval $[0.001, 0.03]$.
The plot of the NMSE (approximated according to~\eqref{eq: error lambda}) as a function of $\lambda$ is given in Figure~\ref{fig: plot lambda} (a). 
The minimum error is obtained for $\lambda\simeq 0.01$, which corresponds to $\sigma_0^2$. This result validates the theory presented in Theorem~\ref{th: MMSE}. 
Actually, when $\lambda$ is small, the estimator interpolates the noisy measurements while, for a large $\lambda$, the estimator tends to oversmooth the curve. The MMSE estimator makes an optimal tradeoff between fitting the data and smoothing the curve.
These observations about $\lambda$ retain their validity for other operators, including noninvertible ones.

\smallskip
\textit{Fourier-Domain Sampling Measurements.}
We consider complex exponential measurement functionals, inducing $\langle \bm{\nu}, s  \rangle = (\widehat{s}[k_1], \ldots , \widehat{s}[k_M])^{\mathsf{T}}$, where the $k_m$ are in $\mathbb{Z}$. We define $\mathcal{N}_{\bm{\nu}} = \{k_m\}_{m=1\ldots M}$, such that $(-k_m) \in \mathcal{N}_{\bm{\nu}}$ for every $k_m \in  \mathcal{N}_{\bm{\nu}}$. 
We consider the measurements $\bm{\nu} = ( e_{k_1}, \ldots , e_{k_M})$. Note that these measurement functionals are complex, which calls for a slight adaptation of the framework presented so far\footnote{One could equivalently consider cosine and sine measurements, to the cost of heavier formulas.}. The noise $\bm{\epsilon} = (\epsilon_1, \ldots, \epsilon_M)$ is then also complex and satisfies the properties:
\begin{itemize}
\item $\epsilon_m=\Re(\epsilon_m)+ \mathrm{j} \ \Im(\epsilon_m)$;
\item $\epsilon_{m_1}=\overline{\epsilon_{m_2}}$, $k_{m_1}=-k_{m_2}$;
\item $\Re(\epsilon_m), \ \Im(\epsilon_m) \sim \mathcal{N}(0,\frac{\sigma_0^2}{2})$, $\forall k_m \neq 0$;
\item $\epsilon_m \in \mathbb{R}$ and $\epsilon_m \sim \mathcal{N}(0,\sigma_0^2)$, $k_m=0$;
\item $\Re(\epsilon_m), \ \Im(\epsilon_m)$ and $\epsilon_{m_1}$, $k_{m_1}=0$, are independent.
\end{itemize}
This means that $\mathbb{E}[\lvert \epsilon_m \rvert^2]=\sigma_0^2$ for every $m$.

We repeated the experiment done with the time-domain sampling using exactly the same procedure and parameters, and $\mathcal{N}_{\bm{\nu}}=\{-2,-1,0,1,2\}$. The experimental curve of the evolution of the NMSE with $\lambda$ is given in Figure~\ref{fig: plot lambda} (b). Again, the minimum is obtained for $\lambda\simeq 0.01=\sigma_0^2$. We now want to compare this curve to the theoretical one.

For the Fourier-sampling case, we were also able to derive the corresponding closed-form formulas for the NMSE~\eqref{eq: error lambda}.

\begin{prop}\label{prop:fouriersampling}
Let $s$ be a Gaussian bridge associated with an invertible operator $\mathrm{L}$, and $y_m = \widehat{s}[k_m] + \epsilon_m$, $m=1\ldots M$, with $k_m \in \mathcal{N}_{\bm{\nu}}$ the sampled frequencies and $\epsilon$ a complex Gaussian noise with variance $\sigma_0^2$ as above.
Then, the MSE of the estimator $\tilde{s}_{\lambda} = \tilde{s}_\lambda(\cdot | \mathbf{y})$ is given by
\begin{equation}\label{eq:FouriersamplingNMSE}
	\mathbb{E}\left[ \lVert s - \tilde{s}_\lambda \rVert_{L_2}^2 \right] = \sum_{m=1}^M \frac{\widehat{h}[k_m] (\lambda^2 + \widehat{h}[k_m] \sigma_0^2)}{(\widehat{h}[k_m]+\lambda)^2} + \sum_{k\notin \mathcal{N}_{\bm{\nu}}} \widehat{h}[k],
\end{equation}
where $h$ is the reproducing kernel of $\mathcal{H}_{\mathrm{L}}$.
\end{prop}

\noindent The proof is given in Appendix~\ref{app:Fouriersampling}. Note that ${\widehat{h}[k]= 1 /\rvert \widehat{L}[k] \lvert^2}$ is real-valued and strictly positive for every $k$. From~\eqref{eq:FouriersamplingNMSE}, we also recover the property that the optimum is reached for $\lambda = \sigma_0^2$ since each of the $M$ terms that appear in the first sum is minimized for this value of $\lambda$.

The theoretical curve for $\mathcal{N}_{\bm{\nu}}=\{-2,-1,0,1,2\}$ is given in Figure~\ref{fig: plot lambda} (b) and is in good agreement with the experimental curve. We explain the slight variation ($0.15 \%$ for the $L_2$-norm over $\lambda\in [0.001,0.03]$) by the fact that~\eqref{eq: error lambda} is only an estimation of the theoretical NMSE.
 \begin{figure}
 \center\hspace*{-0.2cm}
  		\subfigure[{Time-domain sampling}.]{
  \includegraphics[width=0.48\columnwidth]{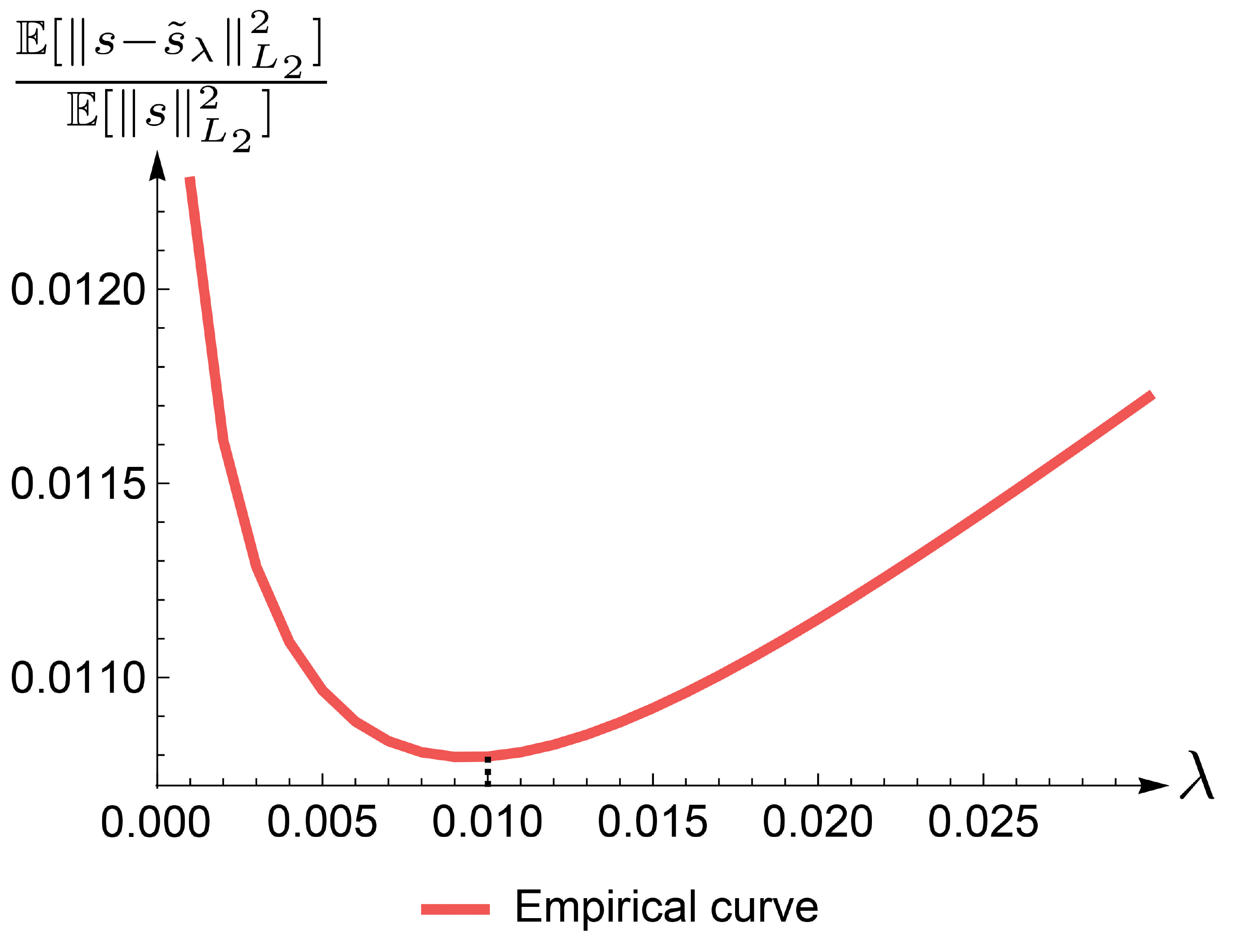}
  }\hspace*{+0.05cm}
  \subfigure[Fourier-domain sampling.]{
    \includegraphics[width=0.48\columnwidth]{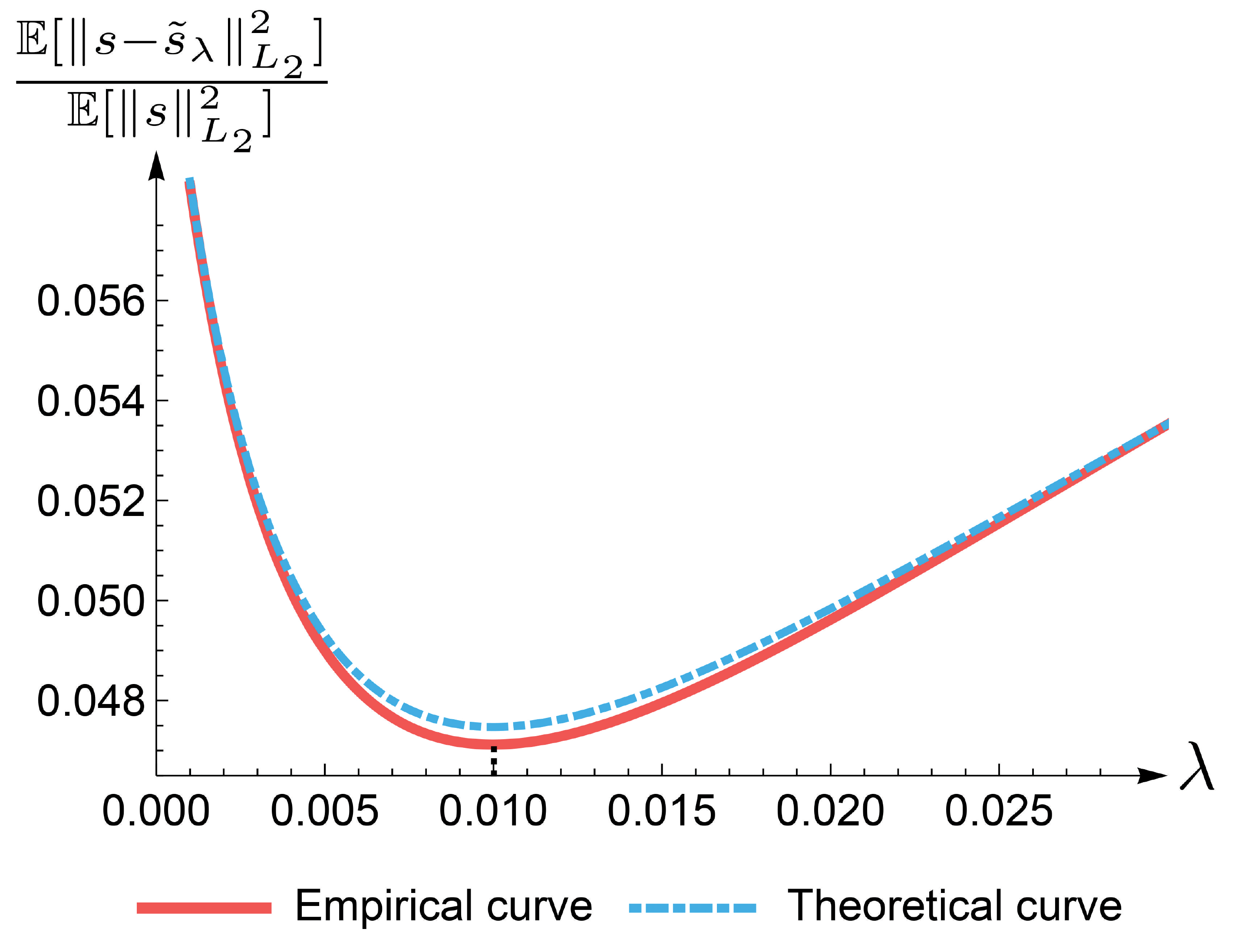}
    }
 \caption{
 Evolution of the NMSE in terms of $\lambda$ for $s \sim \mathcal{GB}(\mathrm{D+I})$ for time and Fourier-domain sampling measurements.
 }
 \label{fig: plot lambda}
 \end{figure}
 
 \subsection{Influence of $\gamma^2$}\label{sec: gamma}
In this section, we only consider noninvertible operators since invertibility has already been addressed in Section~\ref{sec:discussion} (see~\eqref{eq:equivalence}).
 In order to evaluate the specific influence of $\gamma$, we set $\lambda = \sigma_0^2$. Hence, $\tilde{s}_{\gamma, \sigma_0^2} = \tilde{s}_\gamma$. We generated $N=500$ realizations of a Gaussian bridge $s$, and from each one, we extracted $M=30$ noisy measurements. We repeated this for several operators $\mathrm{L}$ and values of $\gamma_0^2$ and $\sigma_0^2$. For each case, we compared $\tilde{s}_{\mathrm{MMSE}}$ to $\tilde{s}_{\gamma \rightarrow 0}$, $\tilde{s}_{\gamma \rightarrow \infty}$, and $f_{\mathrm{RT}}$ in~\eqref{eq:formsolutionRT}, seen here as an additional estimator. The corresponding NMSEs (see~\eqref{eq: error lambda}) are given in Table~\ref{Table: Comparison NMSE}. We make four observations.
 
 \smallskip
1) In each case, the best result is obtained with $\tilde{s}_{\mathrm{MMSE}}$, as expected.
We see, moreover, that
${\lim_{\gamma \rightarrow 0}\mathbb{E}[\| s -\tilde{s}_{\gamma} \|_{L_2}^2] \simeq \mathbb{E}[\| s -f_\mathrm{RT} \|_{L_2}^2]}$.
This is in line with the fact that the functional~\eqref{eq: periodic representer thm} to minimize in Theorem~\ref{th: periodic representer thm} corresponds to~\eqref{eq:L2reguWithGamma} with $\gamma = 0$.

 \smallskip
2) For small values of $\gamma_0^2$ (\textit{i.e.}, $10^{-3}$ or $10^0$), we see that $\mathbb{E}[\| s -f_\mathrm{RT} \|_{L_2}^2] \simeq \mathbb{E}[\| s -\tilde{s}_{\mathrm{MMSE}} \|_{L_2}^2]$. This means that the performances of $\tilde{s}_{\mathrm{MMSE}}$ and $f_{\mathrm{RT}}$ are very similar. This is illustrated in Figure~\ref{fig: plot Table} (a), where $\tilde{s}_{\mathrm{MMSE}}$ and $f_{\mathrm{RT}}$ do coincide. Meanwhile, we see that ${\lim_{\gamma \rightarrow \infty}\mathbb{E}[\| s -\tilde{s}_{\gamma} \|_{L_2}^2] \gg \mathbb{E}[\| s -\tilde{s}_{\mathrm{MMSE}} \|_{L_2}^2]}.$ This is also illustrated in Figure~\ref{fig: plot Table} (a) for $\mathrm{L=D}$. The reconstruction for $\gamma \rightarrow + \infty $ significantly fails to recover the original signal $s$, as the corresponding estimator tends to have zero-mean.

\smallskip
3) For intermediate values of $\gamma_0^2$ (\textit{i.e.}, $\gamma_0^2=10^3$ or $10^6$ according to $\sigma_0$ and the order of the operator), the minimal NMSE is obtained for $\tilde{s}_{\mathrm{MMSE}}$ only. We also observe that ${\mathbb{E}[\| s -f_\mathrm{RT} \|_{L_2}^2] < \lim_{\gamma \rightarrow \infty}\mathbb{E}[\| s -\tilde{s}_{\gamma} \|_{L_2}^2]}.$ This is illustrated in Figure~\ref{fig: plot Table} (b) for $\mathrm{L=D^2+4\pi^2 I}$, $\gamma_0^2=10^6$ and $\sigma_0^2=10^{-4}$, where we can distinguish $\tilde{s}_{\mathrm{MMSE}}$, $\tilde{s}_{\gamma \rightarrow \infty}$, and $f_{\mathrm{RT}}$.
 
\smallskip
4) For large values of $\gamma_0^2$ (\textit{i.e.}, $\gamma_0^2 =10^9$), we observe that ${\lim_{\gamma \rightarrow \infty}\mathbb{E}[\| s -\tilde{s}_{\gamma} \|_{L_2}^2] \simeq \mathbb{E}[\| s -\tilde{s}_{\mathrm{MMSE}} \|_{L_2}^2]}$ and ${\mathbb{E}[\| s -f_{\mathrm{RT}} \|_{L_2}^2] > \mathbb{E}[\| s -\tilde{s}_{\mathrm{MMSE}} \|_{L_2}^2]}.$ In fact, for large $\gamma_0^2$, the Gaussian bridge tends to have vanishing null-space frequencies (with~\eqref{eq: Gaussian bridge}, we have that $\widehat{s}[k_n] = \widehat{w}[k_n] / \gamma_0$ for $n=1 \ldots  N_0$).
Meanwhile, the reconstructed signal $f_{\mathrm{RT}}$ is not constrained to attenuate null-space frequencies. 
The null-space part in~\eqref{eq:formsolutionRT} is mainly responsible for a higher error compared to $\tilde{s}_{\mathrm{MMSE}}$. This is highlighted in Figure~\ref{fig: plot Table} (c).
 
 \smallskip
Observations 2), 3), and 4) suggest the existence of three regimes. 
For further investigation, we present in Figure~\ref{fig: plot gamma} the evolution of NMSE as a function of $\log{\gamma^2}$ for $\mathrm{L=D}$ and $\gamma_0^2=10^0, 10^3$, and $10^6$. 
The minimal error is always obtained for $\gamma^2 \simeq \gamma^2_0$, as predicted by the theory. 
For the three cases, we observe two plateaus: one for $\gamma^2 \in (0,v_1)$ and the other for $\gamma^2 \in (v_2, \infty)$, where $v_1$, $v_2 > 0$. It means that, for each value of $\gamma_0^2$, the estimators $\tilde{s}_{\gamma}$ with $\gamma^2 \in (0,v_1)$ ($(v_2, \infty)$, respectively) are very similar and the reconstruction algorithms are practically indistinguishable. The values of $v_1$ and $v_2$ depend on $\gamma_0^2$. When $\gamma_0^2=10^0$ ($10^6$, respectively), we have that $\gamma_0^2 \in (0,v_1)$ ($(v_2, \infty)$, respectively). However, $\gamma_0^2=10^3 \in [v_1, v_2]$ belongs to none of the plateaus.

Two main conclusions can be drawn from our experiments. 
First, we have strong empirical evidence that 
\begin{equation}
\tilde{s}_\gamma \underset{\gamma\rightarrow 0}{\longrightarrow} f_{\mathrm{RT}},
\end{equation}
which we conjecture to be true for any Gaussian-bridge model.
This is remarkable because it presents the reconstruction based on the periodic representer theorem as a limit case of the statistical approach.
Second, we empirically see that, for reasonably small values of $\gamma_0^2$, the estimators corresponding to $\gamma^2 \leq \gamma_0^2$ are practically indistinguishable from the MMSE estimator. 
This is in particular valid for the representer-theorem reconstruction, for which we then have that
\begin{equation}
f_{\mathrm{RT}} \approx \tilde{s}_{\mathrm{MMSE}}.
\end{equation}

\noindent The variational method is \emph{theoretically} suboptimal to reconstruct Gaussian bridges. However, based on our experiments, it is reasonable to consider this method as \emph{practically} optimal for small values of $\gamma_0^2$ and $\lambda=\sigma_0^2$.
 \begin{table*}
\caption{Comparison of NMSE for $\tilde{s}_{\gamma \rightarrow 0}$, $f_{\mathrm{RT}}$, $\tilde{s}_{\mathrm{MMSE}}$, and $\tilde{s}_{\gamma \rightarrow \infty}$ over $N=500$ iterations. \textbf{Bold}: optimal result.}\label{Table: Comparison NMSE}
\vspace*{-0.3cm}
\begin{center}
\begin{tabularx}{\textwidth}{M{0.9cm}  | M{0.6cm} | Y | Y | Y | Y | Y | Y | Y | Y }
\hline\hline
\multicolumn{2}{Y |}{} & \multicolumn{4}{ c |}{$\sigma_0=10^{-1}$}  & \multicolumn{4}{c }{$\sigma_0=10^{-2}$}\Tstrut\Bstrut \\
\hline
$\mathrm{L}$ & $\gamma_0^2$ & $\tilde{s}_{\gamma \rightarrow 0}$ & $f_{\mathrm{RT}}$ & $\tilde{s}_{\mathrm{MMSE}}$ & $\tilde{s}_{\gamma \rightarrow \infty}$ & $\tilde{s}_{\gamma \rightarrow 0}$ & $f_{\mathrm{RT}}$ & $\tilde{s}_{\mathrm{MMSE}}$ & $\tilde{s}_{\gamma \rightarrow \infty}$\Tstrut\Bstrut\\
\hline
\multirow{5}{*}{$\mathrm{D}$}& $10^{-3}$ & $\mathbf{1.37 \times 10^{-5}}$ & $\mathbf{1.37 \times 10^{-5}}$ & $\mathbf{1.37 \times 10^{-5}}$ & $1.78$ & $\mathbf{8.40 \times 10^{-6}}$ & $\mathbf{8.40 \times 10^{-6}}$ & $\mathbf{8.40 \times 10^{-6}}$ & $2.94$\Tstrut\\
& $10^0$ & $\mathbf{1.17 \times 10^{-2}}$ & $\mathbf{1.17 \times 10^{-2}}$ & $\mathbf{1.17 \times 10^{-2}}$ & $1.66$ & $\mathbf{8.44 \times 10^{-3}}$ & $\mathbf{8.44 \times 10^{-3}}$ & $\mathbf{8.44 \times 10^{-3}}$ & $2.72$\\
& $10^3$ & $1.59 \times 10^{-1}$ & $1.56 \times 10^{-1}$ & $\mathbf{1.49 \times 10^{-1}}$ & $1.58 \times 10^{-1}$ & $1.05 \times 10^{-1}$ & $1.05 \times 10^{-1}$ & $\mathbf{9.96 \times 10^{-2}}$ & $1.21 \times 10^{-1}$\\
& $10^6$ & $1.61 \times 10^{-1}$ & $1.60 \times 10^{-1}$ & $\mathbf{1.43 \times 10^{-1}}$ & $\mathbf{1.43 \times 10^{-1}}$ & $1.07 \times 10^{-1}$ & $1.07 \times 10^{-1}$ & $\mathbf{9.11 \times 10^{-2}}$ & $\mathbf{9.11 \times 10^{-2}}$\\
& $10^9$ & $1.66 \times 10^{-1}$ & $1.66 \times 10^{-1}$ & $\mathbf{1.47 \times 10^{-1}}$ & $\mathbf{1.47 \times 10^{-1}}$ & $1.10 \times 10^{-1}$ & $1.10 \times 10^{-1}$ & $\mathbf{9.34 \times 10^{-2}}$ & $\mathbf{9.34 \times 10^{-2}}$\Bstrut\\
\hline
\multirow{5}{*}{$\mathrm{D^2}$} & $10^{-3}$ & $\mathbf{8.43 \times 10^{-7}}$ & $\mathbf{8.43 \times 10^{-7}}$ & $\mathbf{8.43 \times 10^{-7}}$ & $1.07$ & $3.12 \times 10^{-8}$ & $\mathbf{3.11 \times 10^{-8}}$ & $\mathbf{3.11 \times 10^{-8}}$ & $1.34$\Tstrut\\
& $10^0$ & $9.06 \times 10^{-4}$ & $9.06 \times 10^{-4}$ & $\mathbf{9.05 \times 10^{-4}}$ & $1.07$ & $\mathbf{3.34 \times 10^{-5}}$ & $\mathbf{3.34 \times 10^{-5}}$ & $\mathbf{3.34 \times 10^{-5}}$ & $1.33$\\
& $10^3$ & $4.04 \times 10^{-1}$ & $4.04 \times 10^{-1}$ & $\mathbf{3.61 \times 10^{-1}}$ & $7.1 \times 10^{-1}$ & $\mathbf{1.46 \times 10^{-2}}$ & $\mathbf{1.46 \times 10^{-2}}$ & $\mathbf{1.46 \times 10^{-2}}$ & $5.78 \times 10^{-1}$\\
& $10^6$ & $6.53 \times 10^{-1}$ & $6.53 \times 10^{-1}$ & $\mathbf{3.66 \times 10^{-1}}$ & $\mathbf{3.66 \times 10^{-1}}$ & $2.63 \times 10^{-2}$ & $2.63 \times 10^{-2}$ & $\mathbf{2.26 \times 10^{-2}}$ & $2.29 \times 10^{-2}$\\
& $10^9$ & $6.62 \times 10^{-1}$ & $6.62 \times 10^{-1}$ & $\mathbf{3.86 \times 10^{-1}}$ & $\mathbf{3.86 \times 10^{-1}}$ & $2.65 \times 10^{-2}$ & $2.65 \times 10^{-2}$ & $\mathbf{2.16 \times 10^{-2}}$ & $\mathbf{2.16 \times 10^{-2}}$\Bstrut\\
\hline
\multirow{5}{*}{$\mathrm{D^2+4 I}$} & $10^{-3}$ & $\mathbf{5.53 \times 10^{-7}}$ & $\mathbf{5.53 \times 10^{-7}}$ & $\mathbf{5.53 \times 10^{-7}}$ & $1.03$ & $\mathbf{1.71 \times 10^{-8}}$ & $\mathbf{1.71 \times 10^{-8}}$ & $\mathbf{1.71 \times 10^{-8}}$ & $1.22$\Tstrut\\
& $10^0$ & $5.56 \times 10^{-4}$ & $5.56 \times 10^{-4}$ & $\mathbf{5.55 \times 10^{-4}}$ & $1.04$ & $\mathbf{1.77 \times 10^{-5}}$ & $\mathbf{1.77 \times 10^{-5}}$ & $\mathbf{1.77 \times 10^{-5}}$ & $1.24$\\
& $10^3$ & $3.67 \times 10^{-1}$ & $3.67 \times 10^{-1}$ & $\mathbf{3.04 \times 10^{-1}}$  & $8.79 \times 10^{-1}$ & $1.21 \times 10^{-2}$  & $1.21 \times 10^{-2}$ & $\mathbf{1.20 \times 10^{-2}}$ & $8.71 \times 10^{-1}$\\
& $10^6$ & $1.52$ & $1.52$ & $\mathbf{4.63 \times 10^{-1}}$ & $\mathbf{4.63 \times 10^{-1}}$ & $3.94 \times 10^{-2}$ & $3.94 \times 10^{-2}$ & $\mathbf{2.98 \times 10^{-2}}$ & $3.04 \times 10^{-2}$\\
& $10^9$ & $1.47$ & $1.47$ & $\mathbf{4.87 \times 10^{-1}}$ & $\mathbf{4.87 \times 10^{-1}}$ & $4.67 \times 10^{-2}$ & $4.67 \times 10^{-2}$ & $\mathbf{3.18 \times 10^{-2}}$ & $\mathbf{3.18 \times 10^{-2}}$\Bstrut\\
\hline\hline
\end{tabularx}
\end{center}
\end{table*}
\begin{figure*}\vspace*{-0.15cm}
 	\center\hspace*{-0.1cm}
 	 			\subfigure[$\mathrm{L=D}$, $\gamma_0^2=1$, and $\sigma_0^2=10^{-2}$.]{
 		\includegraphics[width=0.62\columnwidth]{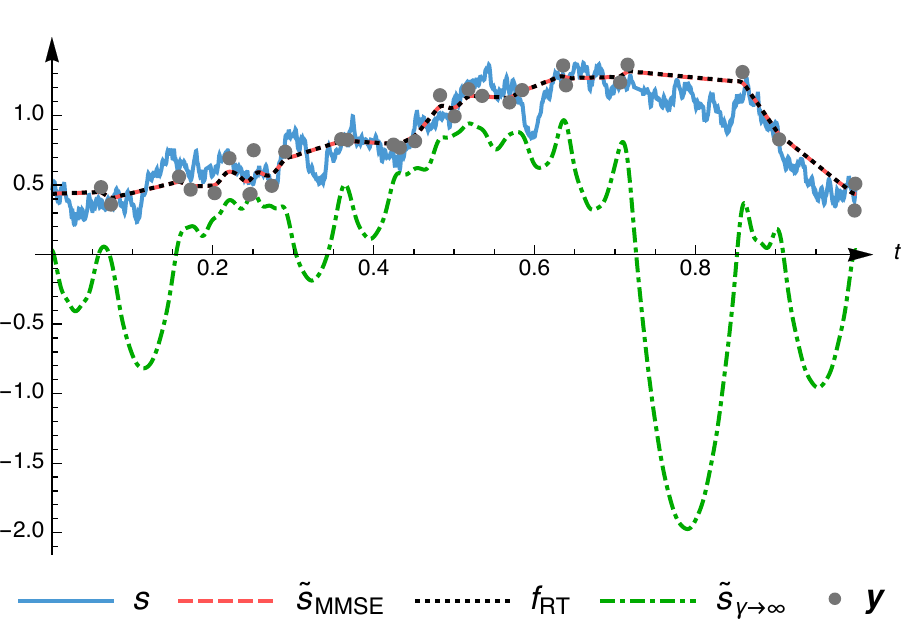}
 		}\hspace*{0.15cm}
 		\subfigure[$\mathrm{L=D^2+4 \pi^2 I}$, $\gamma_0^2=10^6$, and $\sigma_0^2=10^{-4}$.]{
 		\includegraphics[width=0.62\columnwidth]{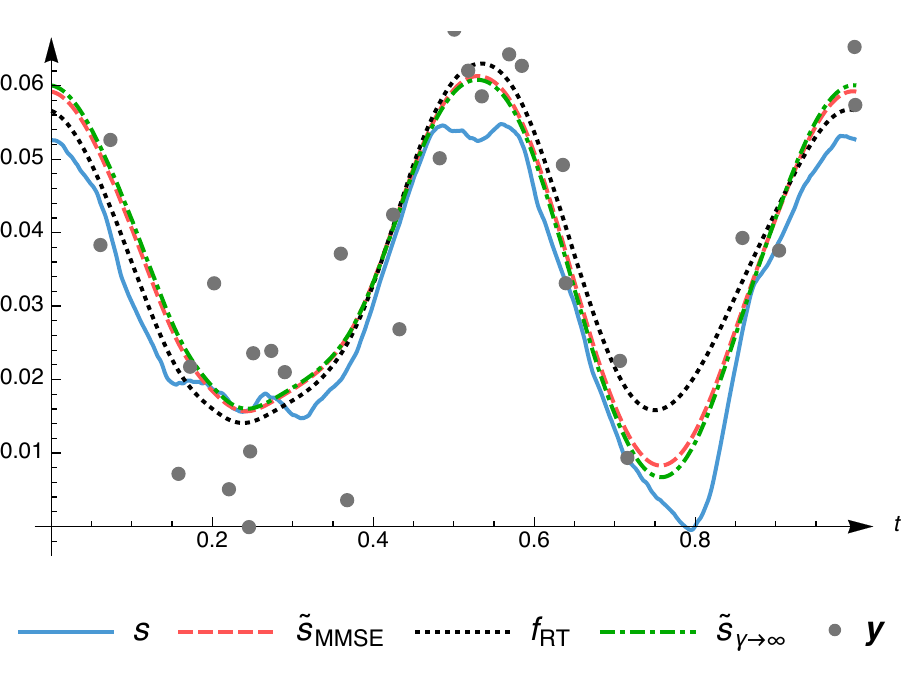}
 		}\hspace*{0.15cm}
 		 \subfigure[$\mathrm{L=D^2}$, $\gamma_0^2=10^9$, and $\sigma_0^2=10^{-4}$.]{
 		\includegraphics[width=0.62\columnwidth]{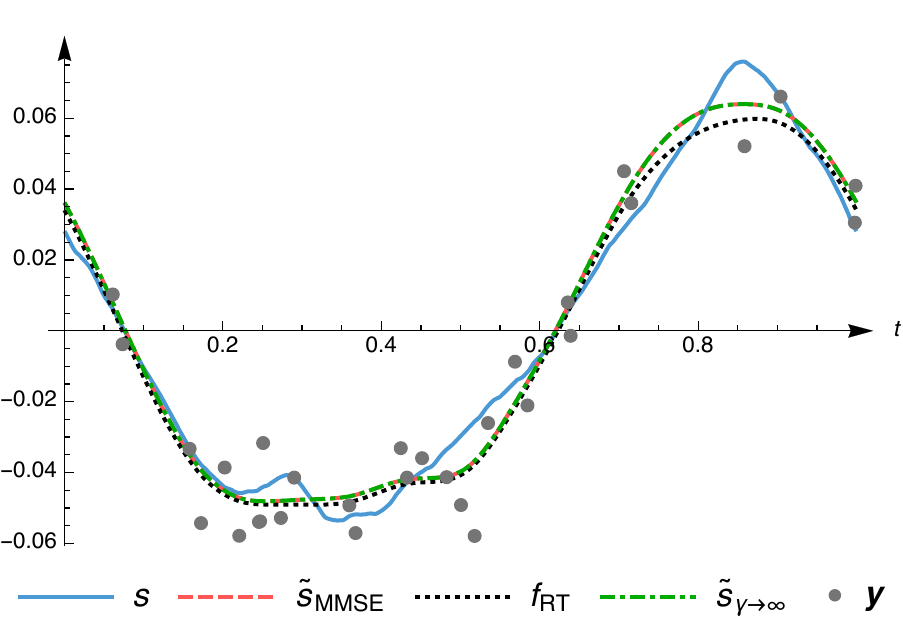}
 		}
 	\caption{Illustrations of $s \sim \mathcal{GB}(\mathrm{L}, \gamma_0^2)$, $\tilde{s}_{\mathrm{MMSE}}$,  $f_{\mathrm{RT}}$, and $\tilde{s}_{\gamma \rightarrow \infty}$ for several operators and values of $\gamma_0^2$ and $\sigma_0^2$ . We used $M=30$ noisy measurements $\mathbf{y}=(y_1, \ldots, y_M)$.}
 	\label{fig: plot Table}\vspace{-0.2cm}
 \end{figure*}
 \begin{figure*}
 \center\hspace*{+0.2cm}
  		\subfigure[$\gamma_0^2=1$.]{
 	\includegraphics[width=0.58\columnwidth]{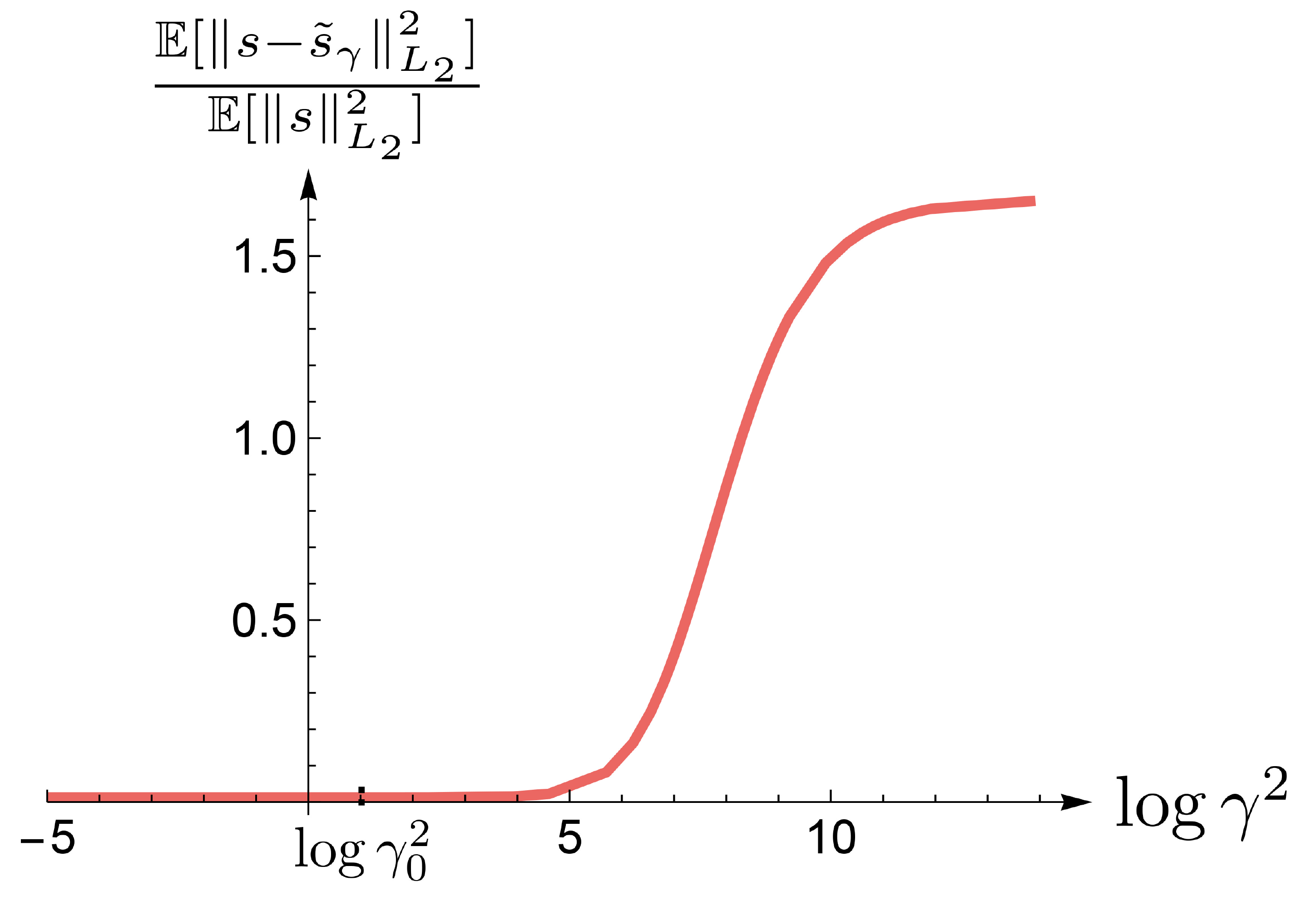}
 		}\hspace*{0.8cm}
		\subfigure[{$\gamma_0^2=10^3$.}]{
 	\includegraphics[width=0.58\columnwidth]{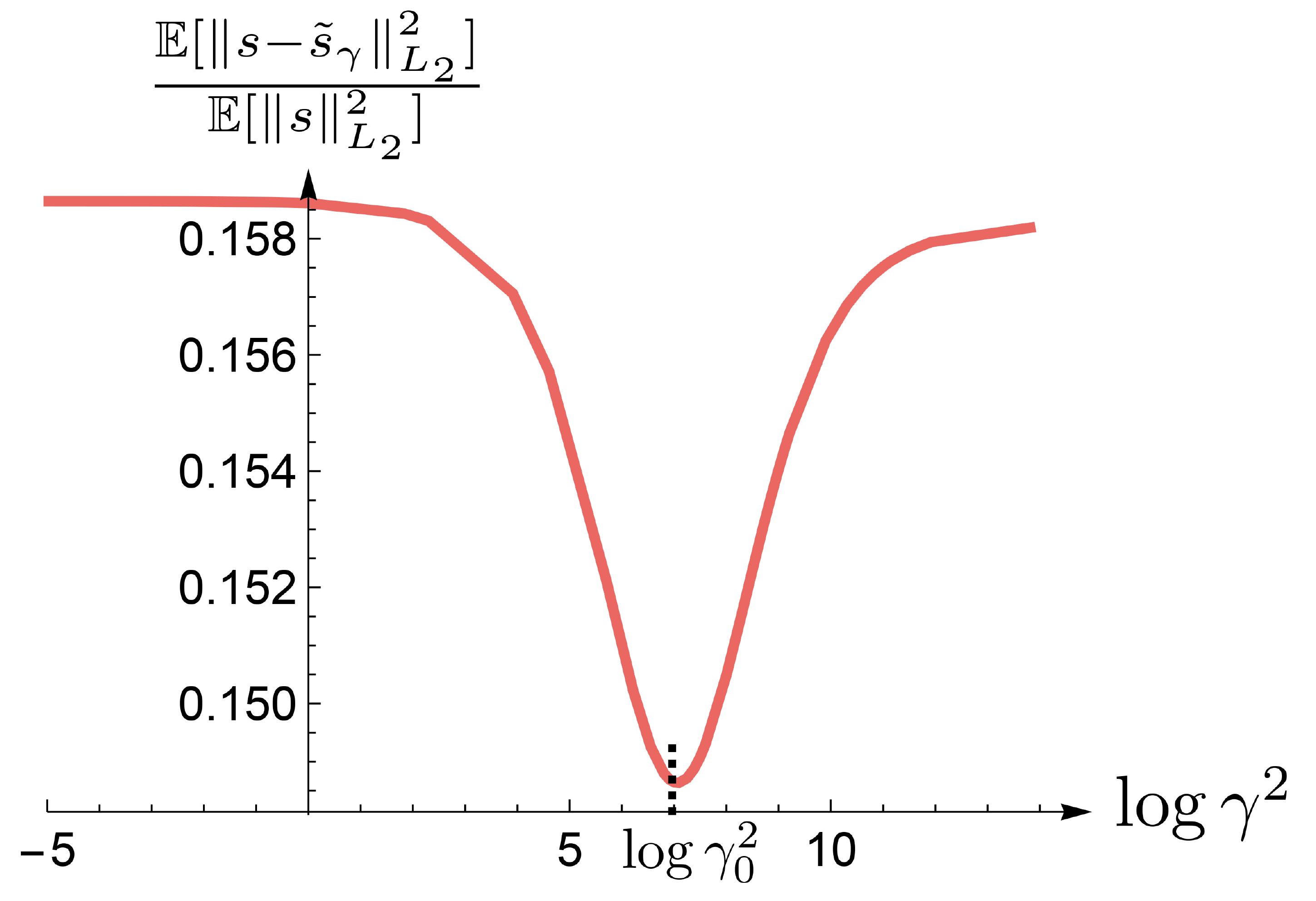}
 		}\hspace*{0.8cm}
 		 		\subfigure[$\gamma_0^2=10^6$.]{
 	\includegraphics[width=0.58\columnwidth]{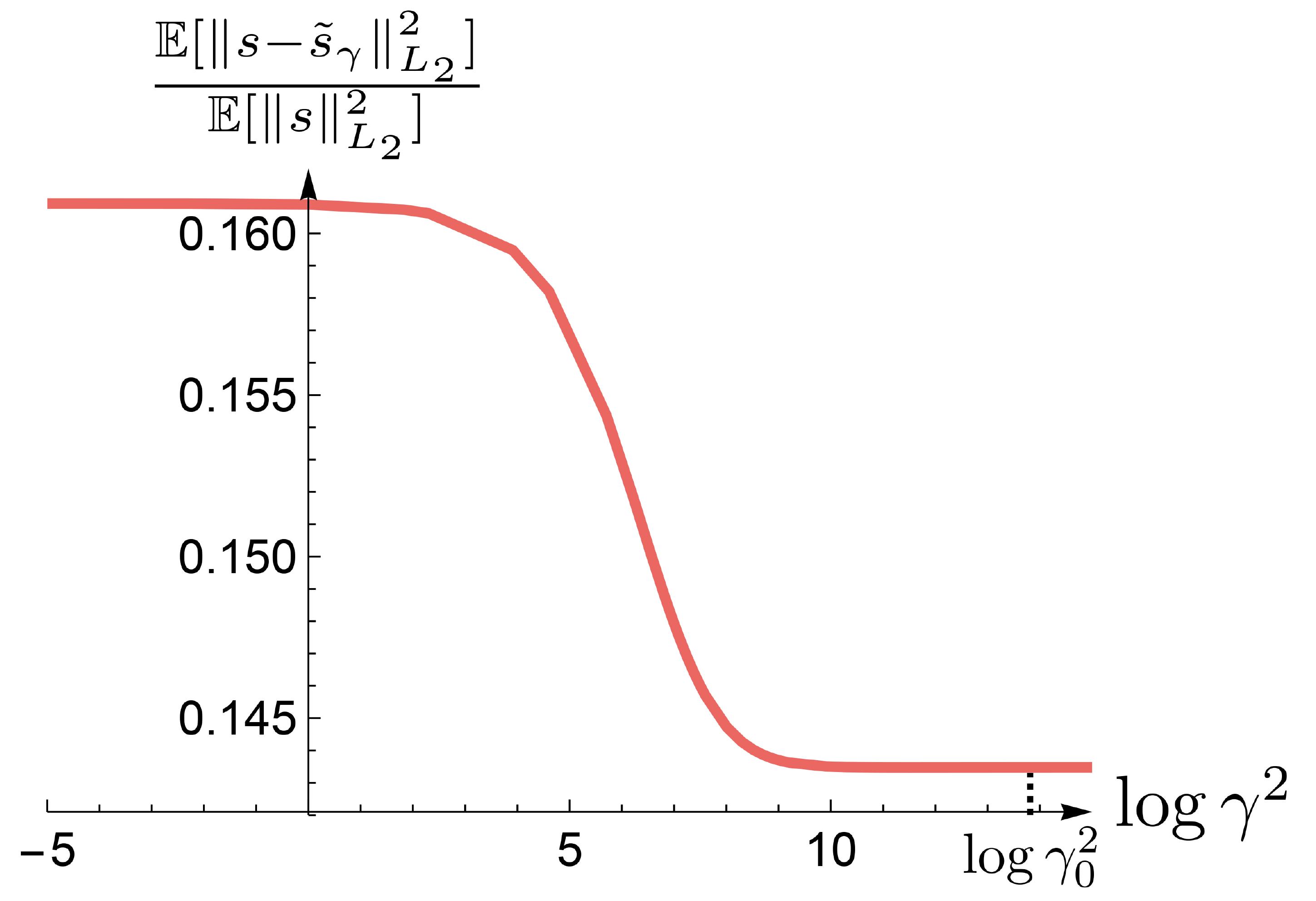}
 		}
 \caption{Evolution of NMSE according to $\gamma$ for $s \sim \mathcal{GB}(\mathrm{D}, \gamma_0^2)$.}
 \label{fig: plot gamma}\vspace{-0.5cm}
 \end{figure*}
 
\section{Discussion}\label{sec:periodicvsline}

\subsection{Comparison with Inverse Problems on the Real Line}
It is worth noting that the periodic setting has important differences as compared to reconstruction methods over the complete real line, which motivated and played an important role in this paper. 
\smallskip
\begin{itemize}[leftmargin=*]
		\item The role of the Dirac impulse $\delta$ is played by the Dirac comb $\Sha$ in the periodic setting. It is indeed the neutral element of the periodic convolution~\eqref{eq: periodic convolution} and appears in the definition of the periodic $\mathrm{L}$-splines (Definition~\ref{def:Lspline}) and RKHS (Definition~\ref{def: RKHS}).
\smallskip
	\item In the real-line setting, in addition to smoothness properties, functions are also characterized by their property of decay at infinity~\cite{Simon2003distributions}. For periodic functions, we only consider the smoothness properties, which brings substantial simplifications. 
\smallskip
	\item In general, a continuous LSI operator does not preserve the asymptotic behavior of the input function. For instance, a test function in the space $\mathcal{S}(\mathbb{R})$ of smooth and rapidly decaying functions is not necessarily mapped to a rapidly decaying function. In contrast, any continuous LSI operator maps the space of periodic test functions $\mathcal{S}(\mathbb{T})$ onto itself (see Section~\ref{subsec:LSIop}). This greatly simplifies the study of operators that act on periodic functions.
\smallskip
	\item The null space of a continuous LSI operator can differ for the two cases. In particular, when acting on periodic functions, the null space of the $n$th derivative $\mathrm{D}^n$ is reduced to constant functions for every $n \geq 1$. This is crucial due to the role of the null space in Theorems~\ref{th: periodic representer thm} and~\ref{th: MMSE}.
\smallskip
\item In Proposition~\ref{prop: RKHS}, we give a necessary and sufficient condition for a continuous LSI operator of finite-dimensional null space to specify a RKHS in the sense of Definition~\ref{def: RKHS}. This is significantly more complicated over the real line, for which only partial results are known~\cite{Unser2016splines}.
\smallskip
\item We have seen that it is not always possible to find a periodic solution $s$ to the equation $\mathrm{L} s = w$, where $w$ is a periodic Gaussian white noise. This lead us to modify the stochastic differential equation (see~\eqref{eq:GB}) and to introduce the family of Gaussian bridges. 	
\smallskip
\item In Theorem~\ref{th: MMSE}, we give the MMSE estimator of the \emph{complete} process $s$, not only for the estimation of $s(t_0)$ at a fixed time $t_0$. In the non-periodic setting, however, solutions of stochastic differential equations are generally not square-integrable. For instance, if $s$ is a nontrivial stationary Gaussian process, then
\begin{align}
	\mathbb{E} [\lVert s \rVert_{L_2(\mathbb{R})}^2] &= \sum_{k\in \mathbb{Z}} \mathbb{E}[\lVert 1_{[k,k+1)}\cdot s \rVert_{L_2(\mathbb{R})}^2]\nonumber \\ 
	& \overset{(i)}{=} \sum_{k\in \mathbb{Z}} \mathbb{E}[\lVert 1_{[0,1)}\cdot s \rVert_{L_2(\mathbb{R})}^2] = \infty,
\end{align}
where $1_{[a,b)}$ is the indicator function on $[a,b)$ and $(i)$ exploits stationarity. Another example is the Brownian motion, whose supremum over $[0,t]$ grows faster than $t^p$ for any $p<1/2$ (almost surely) when $t$ goes to infinity~\cite{Karatzas2012brownian}, hence being of infinite energy.
As a consequence, it is irrelevant to consider the MMSE estimator of the complete process and one ought to, for instance, restrict to MMSE estimators of local values $s(t_0)$ of the process.
\end{itemize}

\subsection{Comparison with TV Regularization}\label{sec:TV}

A recent tendency in the field of signal reconstruction is to rely on sparsity-promoting regularization, motivated by the fact that many real-world signals are sparse in some adequate transform domain~\cite{Elad10,mallat2008wavelet,Unser2014sparse}.

The vast majority of works focuses on the finite-dimensional setting via $\ell_1$-type regularization. However, some authors have recently promoted the reconstruction of infinite-dimensional sparse signals~\cite{Adcock2015generalized,candes2014towards}. The adaptation of discrete $\ell_1$ methods to the continuous domain is based on the total-variation (TV) regularization norm, for which it is possible to derive representer theorems (see~\cite[Theorem 1]{Unser2016splines}). A comparison between Tikhonov and TV variational techniques is proposed in Gupta \textit{et al.}~\cite{gupta2018continuous} for non-periodic signals. In brief, at identical measurements and regularization operator $\mathrm{L}$, Tikhonov regularization favors smooth solutions restricted to a finite-dimensional space, while TV regularization allows for adaptive and more compressible solutions. In~\cite[Table I]{gupta2018continuous}, it was shown on simulations that Tikhonov methods perform better on fractal-type signals, while TV methods are better suited to sparse signals. We expect similar behaviors for the periodic setting.

At the heart of the present paper is the connection between $L_2$-regularization and the statistical formalism of MMSE estimation of Gaussian processes. A theoretical link between deterministic and stochastic frameworks is much harder to provide for sparsity-inducing priors. There is strong empirical evidence that sparse stochastic models are intimately linked to TV-based methods~\cite{Unser2014sparse}, but the extent to which such estimators approach the MMSE solution is still unknown.

\section{Conclusion}\label{sec:conslusions}
We have presented two approaches for the reconstruction of periodic continuous-domain signals from their corrupted discrete measurements.
The first approach is based on optimization theory and culminates with the specification of a periodic representer theorem (Theorem~\ref{th: periodic representer thm}).
In the second approach, a signal is modeled as a stationary periodic random process and the reconstruction problem is transformed into an estimation problem. Theorem~\ref{th: MMSE} then gives the optimal estimator (in the mean-square sense) for Gaussian bridges.

We have also provided theoretical and experimental comparisons of the two approaches and identified two main findings.
First, for invertible operators, the statistical and variational approaches are equivalent and correspond to an identical reconstruction scheme.
For noninvertible operators, however, this equivalence is not valid anymore, but the variational method corresponds to the statistical reconstruction when the parameter $\gamma$ vanishes. More importantly, for small values of $\gamma_0^2$, the variational method is practically equivalent to the optimal statistical reconstruction.
This demonstrates the efficiency of the representer theorem for reconstructing Gaussian bridges, even for noninvertible operators.

\appendix
\subsection{Proof of Proposition~\ref{prop: null space}}\label{app: null space}
The main argument is very classical in the non-periodic setting. We detail it for the sake of completeness and adapt it to the periodic case.

Let $p$ be a function of $\mathcal{N}_{\mathrm{L}}$. As $\mathrm{L}$ is shift-invariant, ${p(\cdot - t_0) \in \mathcal{N}_{\mathrm{L}}}$ for every $t_0 \in \mathbb{T}$. Moreover, $\mathcal{N}_{\mathrm{L}}$ is closed in $\mathcal{S}'(\mathbb{T})$ (as any finite-dimensional linear subspace), thus the first derivative $p'=p^{(1)}$ of $p$ is in $\mathcal{N}_\mathrm{L}$ as the limit of the function $\frac{1}{t_0}( p(\cdot - t_0) - p) \in \mathcal{N}_\mathrm{L}$ when $t_0\rightarrow 0$. We propagate this property to all the derivatives of $p$. 

We now have that $\mathcal{N}_{\mathrm{L}}$ is a finite-dimensional space of dimension $N_0$ and $p^{(k)}\in \mathcal{N}_{\mathrm{L}}$, $\forall k \in [1 \ldots N_0]$. Hence, the family of $(N_0+1)$ functions $p, p^{(1)}, \ldots, p^{(N_0)}$ satisfies an equation of the form $a_{N_0} p^{(N_0)} + \cdots + a_0 p =0$, where $a_k \in \mathbb{C}$ and $(a_0, \ldots , a_{N_0} ) \neq \mathbf{0}$. This implies that $p$, as solution of a differential equation with constant coefficients, is a sum of functions of the form $q(t) \mathrm{e}^{\mu t}$ with $q$ a polynomial and $\mu \in \mathbb{C}$. 

Finally, since we deal with $1$-periodic functions, this constrains $q$ to be a constant function and $\mu = 2 \pi \mathrm{j} k$ with $k \in \mathbb{Z}$. This concludes the proof.

\subsection{Proof of Proposition~\ref{prop: inner product}}\label{app: inner-product}
The linearity, Hermitian symmetry, and non-negativity are easily obtained. We only need to verify that ${\|f\|_{\mathcal{H}_{\mathrm{L}}}=\langle f,f \rangle_{\mathcal{H}_{\mathrm{L}}}^{\frac{1}{2}}=0 \Leftrightarrow f=0}$. For this, we observe that
\begin{align}
\langle f,f \rangle_{\mathcal{H}_{\mathrm{L}}}=0 &\Leftrightarrow \int_0^1|\mathrm{L}f(t)|^2\mathrm{d}t+\gamma^2\sum_{n=1}^{N_0}|\widehat{f}[k_n]|^2=0 \nonumber \\
&\Leftrightarrow \sum_{k \in \mathcal{K}_{\mathrm{L}}} \big | \widehat{f}[k] \big |^2 
\underbrace{\big |\widehat{L}[k] \big |^2 }_{\neq 0}
+\gamma^2\sum_{n=1}^{N_0}|\widehat{f}[k_n]|^2=0,
\end{align}
which implies that $\widehat{f}[k]=0$ for all $k \in \mathbb{Z}$. Hence, ${\langle f,f \rangle_{\mathcal{H}_{\mathrm{L}}}=0 \Leftrightarrow f=0}$.

\subsection{Proof of Proposition~\ref{prop: RKHS}}\label{app: RKHS}
For the proof, we set $A = \sum\limits_{k \in \mathcal{K}_{\mathrm{L}}}\frac{1}{|\widehat{L}[k]|^2}$. The Hilbert space $\mathcal{H}_{\mathrm{L}}$ is a RKHS if and only if $\Sha \in \mathcal{H}_{\mathrm{L}}'$ or, equivalently, if there exists $C>0$ such that
\begin{equation}\label{eq: RKHS}
\forall f \in \mathcal{S}(\mathbb{T}), \ \ | \langle \Sha, f \rangle| \leq C \| f \|_{\mathcal{H}_{\mathrm{L}}}.
\end{equation}

 Assume that $A < + \infty$. Let $c$ be the sequence such that $c[k] = 1/\widehat{L}[k]$ if $k\in \mathcal{K}_{\mathrm{L}}$ and $c[k] = 1 / \gamma$ otherwise. 
Using the Cauchy-Schwarz inequality, we have, for every $f \in \mathcal{S}(\mathbb{T})$, that
 \begin{align}
 \langle \Sha, f \rangle^2 =  \left( \sum \widehat{f}[k] \right)^2 &\leq  {\bigg(\sum  |c[k]|^2\bigg) \bigg(\sum  \bigg |\frac{\widehat{f}[k]}{c[k]}\bigg \vert^2}\bigg) \nonumber \\
&=( {N_0}/{\gamma^2 }+A) \| f \|^2_{\mathcal{H}_{\mathrm{L}}}.
\end{align}
Hence,~\eqref{eq: RKHS} is satisfied for $C=({N_0}/{\gamma^2 }+A)^{1/2}> 0$.
For the converse, we define $f_m \in \mathcal{S}(\mathbb{T})$ such that $$\widehat{f_m}[k]=\begin{cases}
0, & \mbox{if } |k|>m \mbox{ or } k=k_n, n \in [1 \ldots N_0]\\
\frac{1}{ |\widehat{L}[k]|^2}, & \mbox{otherwise.}
\end{cases}$$
Then, we readily observe that $\lim\limits_{m\rightarrow + \infty}\frac{ \ | \langle \Sha, f_m \rangle|}{\| f_m\|_{\mathcal{H}_{\mathrm{L}}}}= \sqrt{A}.$
Therefore, as soon as $A=+\infty$, $\langle \Sha,f \rangle / \lVert f \rVert_{\mathcal{H}_{\mathrm{L}}}$ is not bounded in $\mathcal{S}(\mathbb{T})$ and $\mathcal{H}_{\mathrm{L}}$ is not a RKHS. 

The reproducing kernel is characterized by the relation
$f(\tau) = \langle h(\cdot ,\tau ), f \rangle_{\mathcal{H}_{\mathrm{L}}}$ for every $f \in \mathcal{H}_{\mathrm{L}}$. Let $\mathrm{R}$ be the operator, often called the Riesz map, such that $\langle \mathrm{R} g , f \rangle_{\mathcal{H}_{\mathrm{L}}} = \langle g, f \rangle$ for any $f \in \mathcal{H}_{\mathrm{L}}$ and $g \in \mathcal{H}'_{\mathrm{L}}$. Then, $h(\cdot ,\tau) = \mathrm{R} \{ \Sha( \cdot - \tau )\}$. 
Moreover, we have that $\langle \mathrm{R} e_k ,e_m \rangle_{\mathcal{H}_{\mathrm{L}}} = \delta[k-m]$. In addition, 
\begin{small}
\begin{align}
\langle \mathrm{R}e_k, e_m\rangle _{\mathcal{H}_{\mathrm{L}}}&=\langle \mathrm{L}\mathrm{R}e_k, \mathrm{L} e_m\rangle  + \gamma^2 \sum_{n=1}^{N_0}\widehat{\mathrm{R}e_k}[k_n]\overline{\widehat{e_m}[k_n]} \nonumber \\
&=\langle\mathrm{R}e_k,  \mathrm{L}^{\ast}\mathrm{L} e_m\rangle  + \gamma^2 \sum_{n=1}^{N_0}\widehat{\mathrm{R}e_k}[k_n]\delta[m-k_n] \nonumber \\
&=|\widehat{L}[m]|^2\widehat{\mathrm{R}e_k}[m]+ \gamma^2 \sum_{n=1}^{N_0}\widehat{\mathrm{R}e_k}[k_n]\delta[m-k_n].
\end{align}
\end{small}
Hence, $\mathrm{R}$ is characterized for $k,m\in \mathbb{Z}$ by the relation
\begin{equation}\label{eq: 1}
|\widehat{L}[m]|^2\widehat{\mathrm{R}e_k}[m]+\gamma^2 \sum_{n=1}^{N_0}\widehat{\mathrm{R}e_k}[k_n]\delta[m-k_n]=\delta[k-m].
\end{equation}
For $k \in \mathcal{K}_{\mathrm{L}}$, we deduce from~\eqref{eq: 1} that $\widehat{\mathrm{R}e_k} [m] = 1 /\lvert \widehat{L}[k]\rvert^2$ if $m=k$ and $0$ otherwise. We also deduce that, for $k=k_n$, $\widehat{\mathrm{R}e_{k_n}} [m]= 1/\gamma^2$ if $m=k_n$ and $0$ otherwise.
Thus, $\mathrm{R}$ is shift-invariant ($\widehat{\mathrm{R}e_k}[m] = 0$ for every $m\neq k$), meaning that $h(t,\tau)$ depends only on $(t-\tau)$. 
Moreover, the Fourier multiplier of $\mathrm{R}$, which is also the discrete Fourier transform of $h_\gamma (t) = h(t,0)$, is $\widehat{R}[k]= 1 /\lvert \widehat{L}[k]\rvert^2$ if $k\in \mathcal{K}_{\mathrm{L}}$ and $1/\gamma^2$ if $k = k_n$. This is equivalent to~\eqref{eq: reproducing kernel} and concludes the proof.

\subsection{Proof of Theorem~\ref{th: periodic representer thm}}\label{app: per rep th}
To prove Theorem~\ref{th: periodic representer thm}, we first show that the optimization problem~\eqref{eq: periodic representer thm} has a unique solution by convex-optimization arguments. Then, we connect this solution to the abstract representer theorem (see for instance~\cite[Theorem 16.1]{Wendland2004scattered}) to deduce the form of the solution. We start with some preliminary results for the first part. 

\begin{lemma}\label{lemma:convcoerc}
Under the condition of Theorem~\ref{th: periodic representer thm}, the functional $\phi: \ {\mathcal{H}_{\mathrm{L}}} \rightarrow \mathbb{R}^+ $ defined by
$\phi(f)=F(\mathbf{y}, \langle \boldsymbol\nu , f \rangle)+\lambda \| \mathrm{L}f\| _{L_2}^2$
 is strictly convex and coercive, meaning that $\phi(f) \rightarrow \infty$ when $\lVert f \rVert_{\mathcal{H}_{\mathrm{L}}} \rightarrow \infty$. 
\end{lemma} 

\begin{proof}
\textit{Strict convexity:} $\phi$ is convex as a sum of two convex functions. For the strict convexity, we fix $\mu \in  (0,1)$ and ${f,g \in \mathcal{H}_{\mathrm{L}}}$. It is then sufficient to show that the equality 
\begin{equation} \label{eq:equalityconvexity}
 \phi( \mu f + (1-\mu) g ) = \mu \phi(f) + (1-\mu) \phi(g)
\end{equation}
implies that $f=g$. The functions $F( \mathbf{y} , \boldsymbol{\nu} \{ \cdot \} )$ and $\lVert \mathrm{L} \cdot \rVert_{L_2}$ are convex, therefore~\eqref{eq:equalityconvexity} together with the linearity of both $\boldsymbol{\nu}$ and $\mathrm{L}$ implies the two relations
\small
\begin{align}
	F( \mathbf{y} , \mu \nu(f) + (1-\mu) \nu(g) ) &=  \mu 	F( \mathbf{y} ,\nu(f)  )  + (1 - \mu) 	F( \mathbf{y} ,  \nu(g) ) \nonumber \\
	\lVert \mu \mathrm{L} f + (1-\mu) \mathrm{L} g \rVert_{L_2}^2 &= \mu \lVert \mathrm{L} f \rVert_{L_2}^2 + (1 - \mu ) \lVert \mathrm{L} g \rVert_{L_2}^2.
\end{align}
\normalsize
Now, taking advantage of the strict convexity of $F(\mathbf{y}, \cdot)$ and $\lVert \cdot \rVert_{L_2}^2$, we deduce that $\boldsymbol{\nu}(f) = \boldsymbol{\nu}(g)$ and $\mathrm{L} f = \mathrm{L} g$. This means, in particular, that $(f-g)$ is in the intersection of the null spaces of $\boldsymbol{\nu}$ and $\mathrm{L}$, assumed to be trivial. Finally, $f=g$ as expected.
\smallskip

\textit{Coercivity:}
The measurement functional $\boldsymbol{\nu}$ is linear and continuous, hence there exists $A>0$ such that ${\lVert \langle \nu , f \rangle \rvert^2 \leq A \lVert f \rVert_{\mathcal{H}_{\mathrm{L}}}^2}$ for any $f \in \mathcal{H}_{\mathrm{L}}$. Moreover, since $\boldsymbol{\nu}$ is injective and linear when restricted to the finite-dimensional null space $\mathcal{N}_{\mathrm{L}}$, there exists $B>0$ such that $\lVert \langle \nu ,p \rangle \rVert^2 \geq B \lVert p \rVert_{\mathcal{H}_{\mathrm{L}}}^2$ for any $p\in\mathcal{N}_{\mathrm{L}}$. Any $f\in \mathcal{H}_{\mathrm{L}}$ can be decomposed uniquely as 
\begin{equation}
f = \sum_{k\in\mathcal{K}_{\mathrm{L}}} \widehat{f}[k] e_k + \sum_{n=1}^{N_0} \widehat{f}[k_n] e_{k_n}= g + p.
\end{equation}
In that case, we easily see that $\lVert g \rVert_{\mathcal{H}_{\mathrm{L}}} = \lVert \mathrm{L} f \rVert_{L_2}$. In particular, we deduce that
\begin{align} \label{eq:fwithLnu}
	\lVert f \rVert_{\mathcal{H}_{\mathrm{L}}}^2  & = \lVert g \rVert_{\mathcal{H}_{\mathrm{L}}}^2  + \lVert p\rVert_{\mathcal{H}_{\mathrm{L}}}^2  \leq \lVert \mathrm{L} f \rVert_{L_2}^2 + \frac{1}{B} \lVert \langle \boldsymbol{\nu}, p \rangle \rVert^2 \nonumber\\
	&\leq  \lVert \mathrm{L} f \rVert_{L_2}^2 + \frac{1}{B} \left( \lVert \langle \boldsymbol{\nu}, f \rangle \rVert + \lVert \langle \boldsymbol{\nu}, g \rangle \rVert \right)^2 \nonumber\\
	&\leq  \lVert \mathrm{L} f \rVert_{L_2}^2 + \frac{1}{B} \left( \lVert \langle \boldsymbol{\nu}, f \rangle \rVert + A^{1/2} \lVert \mathrm{L} f \rVert_{L_2} \right)^2 \nonumber \\
	& \leq C \left( \lVert \mathrm{L} f \rVert_{L_2}^2 +  \lVert \langle \boldsymbol{\nu}, f \rangle \rVert^2\right)
\end{align}
for $C>0$ large enough.
Now, consider a sequence of functions $f_m \in \mathcal{H}_{\mathrm{L}}$ such that $	\lVert f \rVert_{\mathcal{H}_{\mathrm{L}}} \rightarrow \infty$. We want to show that, for $m$ large enough, $\phi(f_m)$ is arbitrarily large. Due to~\eqref{eq:fwithLnu}, for $m$ large enough, $\lVert \mathrm{L} f_m \rVert_{L_2}$ or $\lVert \langle \boldsymbol{\nu}, f_m \rangle \rVert$ are arbitrarily large. The former implies obviously that $\phi(f_m)$ can be made as large as we want. It is also true for the latter because ${\phi(f_m) \geq F(\mathbf{y}, \langle \boldsymbol{\nu} , f_m\rangle)}$ and $F$ is coercive. 
This means that $\phi(f_m)$ goes to infinity when $m\rightarrow \infty$, hence $\phi$ is coercive. 
\end{proof}

\noindent As $\phi$ is a strictly convex and coercive functional (Lemma~\ref{lemma:convcoerc}), the optimization problem~\eqref{eq: periodic representer thm} has the unique solution $f_{\mathrm{RT}}$. We denote $z_0=\langle \boldsymbol\nu , f_{\mathrm{RT}}\rangle$. The function $f_{\mathrm{RT}}$ can be uniquely decomposed as 
\begin{equation}
f_{\mathrm{RT}} = \sum_{k\in \mathcal{K}_{\mathrm{L}}} \widehat{f}_{\mathrm{RT}}[k] e_k + \sum_{n=1}^{N_0} \widehat{f}_{\mathrm{RT}}[k_n] e_{k_n} = g_{\mathrm{RT}} + p_{\mathrm{RT}}.
\end{equation}

We recall the abstract representer theorem. This result can be found in~\cite[Theorem 8]{gupta2018continuous} with a formulation close to ours.
\begin{prop} \label{prop:ART}
Let $\mathcal{H}$ be a Hilbert space, $\bm{\nu}=  (\nu_1, \ldots , \nu_M)$ be a vector of $M$ linear and continuous measurement functionals over $\mathcal{H}$, and $\mathbf{y}_0 \in \mathbb{R}^M$. There exists a unique minimizer of the optimization problem
\begin{equation}
\min_{f \in \mathcal{H}} \lVert f \rVert_\mathcal{H} \ \text{s.t.} \ \bm{\nu} = \mathbf{y
}_0,
\end{equation}
which is of the form $f_{\mathrm{opt}} = \sum\limits_{m=1}^M a_m \mathrm{R} \nu_m$, where $a_m \in \mathbb{R}$ and $\mathrm{R} : \mathcal{H}' \rightarrow \mathcal{H}$ is the Riesz map of $\mathcal{H}$.
\end{prop}
We consider the Hilbert space ${\widetilde{\mathcal{H}}_{\mathrm{L}} = \{ f \in \mathcal{H}_{\mathrm{L}}, \ \mathrm{Proj}_{\mathcal{N}_{\mathrm{L}}} \{f\} = 0}\}$, on which $\rVert \mathrm{L} f \rVert_{L_2}$ is a Hilbertian norm. The linear measurements $\nu_m$ are in the dual space $\widetilde{\mathcal{H}}_{\mathrm{L}}'$, once restricted as linear functionals on $\widetilde{\mathcal{H}}_{\mathrm{L}}$. The interpolation constraint is chosen as $\mathbf{y}_0 = \mathbf{z}_0 - \bm{\nu}(p_{\mathrm{RT}})$. 
Applying Proposition~\ref{prop:ART} to this case, we deduce that there exists a unique minimizer
\begin{equation}
h_{\mathrm{opt}} = \underset{h \in \widetilde{\mathcal{H}}_{\mathrm{L}}, \boldsymbol{\nu} (h) = \mathbf{y}_0}{\arg \min} \lVert  \mathrm{L} h \rVert_{L_2}
\end{equation}
which is of the form $h_{\mathrm{opt}} = \sum_{m=1}^M a_m \mathrm{R} \nu_m$, 
$\mathrm{R}$ being the Riesz map between $\widetilde{\mathcal{H}}_{\mathrm{L}}'$ and $\widetilde{\mathcal{H}}_{\mathrm{L}}$. In our case, the function $\mathrm{R}\nu_m$ is given by $\mathrm{R}\nu_m = \sum_{k \in \mathcal{K}_{\mathrm{L}}} \frac{\widehat{\nu}_m[k]}{\lvert \widehat{L}[k] \rvert^2} e_k$. In particular, one easily sees from the expression of $\varphi_m$ that it satisfies
\begin{equation}
\mathrm{R}\nu_m = \varphi_m - \gamma^2 \mathrm{Proj}_{\mathcal{N}_{\mathrm{L}}} \{\nu_m\}.
\end{equation}

Moreover, we have that $h_{\mathrm{opt}} = g_{\mathrm{RT}}$. Indeed, $g_{\mathrm{RT}}$ is clearly among the functions $h$ over which one minimizes and one cannot have that $\lVert  \mathrm{L} h_{\mathrm{opt}} \rVert_{L_2} < \lVert  \mathrm{L} g_{\mathrm{RT}} \rVert_{L_2}$ (otherwise, the function ${f  = h_{\mathrm{opt}} + p_{\mathrm{RT}}}$ would be a minimizer of~\eqref{eq: periodic representer thm} different from $f_{\mathrm{RT}}$, which is impossible).
Putting things together, we get that
\begin{align}
	f_{\mathrm{RT}} &= g_{\mathrm{RT}} + p_{\mathrm{RT}}  = \sum_{m=1}^M a_m \mathrm{R}\nu_m + p_{\mathrm{RT}} \nonumber  \\
		&= \sum_{m=1}^M a_m \varphi_m - \gamma^2  \sum_{m=1}^M a_m \mathrm{Proj}_{\mathcal{N}_{\mathrm{L}}} \{\nu_m\}+ p_{\mathrm{RT}}.
\end{align}
Since $(- \gamma^2  \sum_{m=1}^M a_m \mathrm{Proj}_{\mathcal{N}_{\mathrm{L}}} \{\nu_m\}+ p_{\mathrm{RT}})$ is in the null space of $\mathrm{L}$, it can be developed as $\sum_{n=1}^{N_0} b_n e_{k_n}$, giving~\eqref{eq:formsolutionRT}.

The last ingredient is to remark that $a_m$ satisfies $\mathbf{P}^{\mathsf{T}} \mathbf{a} = \mathbf{0}$. This comes from the fact that, by construction, ${\sum a_m \mathrm{R}\nu_m  \in \widetilde{\mathcal{H}}_{\mathrm{L}}'}$ and, by applying the Riesz map, ${\sum a_m \nu_m \in \widetilde{\mathcal{H}}_{\mathrm{L}}}$, meaning that the projection of this element into the null space is zero. This is precisely equivalent with the expected condition.

\subsection{Proof of Proposition~\ref{prop: interpolation pb}}\label{app: interpolation pb}
We compute~\eqref{eq: periodic representer thm} for $F$ the quadratic cost function. We have that $f_{\mathrm{RT}} = \sum_{m=1}^{M} a_m\varphi_m + \sum_{n=1}^{N_0} b_n e_{k_n}$, as given by~\eqref{eq:formsolutionRT}.
It then suffices to find the optimal vectors $\mathbf{a}$ and $\mathbf{b}$. We therefore rewrite~\eqref{eq: periodic representer thm} in terms of these two vectors.

From simple computations, we have, with the notations of Proposition~\ref{prop: interpolation pb}, that $\langle \boldsymbol{\nu} ,  \sum_{n=1}^{N_0} b_n e_{k_n} \rangle = \mathbf{P} \mathbf{b} $ and 
$\langle \boldsymbol{\nu} , \sum_{m=1}^M a_m \varphi_m \rangle = \mathbf{G} \mathbf{a},$
where we used for the latter that $\mathbf{G}_{m_1,m_2} = \langle \nu_{m_1}, h_\gamma * \nu_{m_2} \rangle = \langle \nu_{m_1} , \varphi_{m_2} \rangle$.
Hence, 
\begin{equation}
\lVert \mathbf{y} - \langle \boldsymbol{\nu} , f \rangle \rVert ^2= \lVert \mathbf{y} - \mathbf{G} \mathbf{a} - \mathbf{P} \mathbf{b} \rVert^2.
\end{equation}

From the definition of $h_{\gamma}$ in~\eqref{eq: reproducing kernel}, we see that ${(\mathrm{L}^* \mathrm{L} h_{\gamma}) * f = f}$ for every $f$ whose Fourier coefficients $\widehat{f}[k_n]$ do vanish for every $n=1\ldots N_0$.
Now, the relation $\overline{\mathbf{P}}^{\mathsf{T}} \mathbf{a} = \mathbf{0}$ in Theorem~\ref{th: periodic representer thm} shows precisely that $\sum_{n=1}^M a_m \nu_m$ satisfies this property. 
In particular, we deduce that
\begin{equation}
\mathrm{L}^* \mathrm{L} \left\{\sum_{m=1}^M a_m \varphi_m\right\} = (\mathrm{L}^* \mathrm{L} h_{\gamma}) * \sum_{m=1}^{M} a_m \nu_m = \sum_{m=1}^{M}a_m \nu_m.
\end{equation}
As a consequence, we have that
\begin{align}
\lVert \mathrm{L} f_{\mathrm{RT}} \rVert^2_{L_2} &=   \langle \mathrm{L}^* \mathrm{L} \sum_{m_1=1}^M a_{m_1} \varphi_{m_1}, \sum_{m_2=1}^M a_{m_2} \varphi_{m_2} \rangle \nonumber \\ &=\sum_{m_1=1}^M \sum_{m_2=1}^M a_{m_1} \mathbf{G}_{m_1,m_2} a_{m_2} = ( \mathbf{G} \mathbf{a})^{\mathsf{T}} \mathbf{a}.
\end{align}
Finally, one has that 
\begin{equation}
\lVert \mathbf{y} - \langle \boldsymbol{\nu}, f_{\mathrm{RT}} \rangle \rVert^2 + \lambda \lVert \mathrm{L}f_{\mathrm{RT}} \rVert_{L_2}^2 = \lVert \mathbf{y} - \mathbf{G} \mathbf{a} - \mathbf{P} \mathbf{b} \rVert^2  + \lambda  ( \mathbf{G} \mathbf{a})^{\mathsf{T}} \mathbf{a}.
\end{equation}

\noindent By computing the partial derivatives, we find that the vectors $\mathbf{a}$ and $\mathbf{b}$ are given by~\eqref{eq:matrixform}.

\subsection{Proof of Proposition~\ref{prop: Spline}}\label{app: Spline}
Since $\nu_m = \Sha(\cdot - t_m)$, the form of the solution~\eqref{eq:formsolutionRT} is 
$f_{\mathrm{RT}}(t)=\sum\limits_{m=1}^M a_m h_{\gamma}(t-t_m)+\sum\limits_{n=1}^{N_0}b_n e_{k_n}(t)$. We have moreover that
 $\mathbf{P}^{\mathsf{T}} \mathbf{a} = \mathbf{0}$, where ${[\mathbf{P}]_{m,n} =\mathrm{e}^{\mathrm{j} 2\pi k_n t_m}}$. 
From~\eqref{eq: reproducing kernel}, we then deduce that 
$\mathrm{L}^{\ast}\mathrm{L} \{h_{\gamma}\}(t)  = \sum_{k \in \mathcal{K}_{\mathrm{L}}} \lvert \widehat{L}[k] \rvert^2 \frac{e_k(t)}{\lvert \widehat{L}[k] \rvert^2}  = \bigg (\Sha(t) - \mathrm{Proj}_{\mathcal{N}_{\mathrm{L}}} \{\Sha\} (t)\bigg ).$
By linearity, we get that
\begin{small}
\begin{align}
&\mathrm{L}^{\ast}\mathrm{L} \{f_{\mathrm{RT}}\}(t)
=
\sum_{m=1}^M a_m \mathrm{L}^{\ast}\mathrm{L} \{h_{\gamma}\}(t - t_m)  \nonumber
 \\
&= 
\sum_{m=1}^M a_m \Sha(t - t_m)    - 
\sum_{m=1}^M a_m \mathrm{Proj}_{\mathcal{N}_{\mathrm{L}}} \{\Sha( \cdot - t_m) \} (t) \nonumber \\
&= 
\sum_{m=1}^M a_m \Sha(t - t_m)    -  \sum_{n=1}^{N_0} \sum_{m=1}^M a_m \mathrm{e}^{-\mathrm{j} 2\pi k t_m}  e_{k_n} \nonumber \\
&= \sum_{m=1}^M a_m \Sha(t - t_m) - \sum_{n=1}^{N_0} [\overline{\mathbf{P}}^\mathsf{T}  \mathbf{a}]_n e_{k_n}   \label{eq:plouf1} \\
&= \sum_{m=1}^M a_m \Sha(t - t_m), \label{eq:plouf2}
\end{align}
\end{small}
\noindent where we used that $[\overline{\mathbf{P}}]_{m,n} = \mathrm{e}^{-\mathrm{j} 2\pi k t_m}$ in~\eqref{eq:plouf1} and that $\overline{\mathbf{P}}^\mathsf{T}  \mathbf{a} = \overline{\mathbf{P}^\mathsf{T} \mathbf{a}} = \mathbf{0}$ in~\eqref{eq:plouf2}.
Finally, $f_{\mathrm{RT}}$ is a periodic $(\mathrm{L}^*\mathrm{L})$-spline with weights $a_m$ and knots $t_m$. 

\subsection{Proof of Proposition~\ref{prop:covariancebridge}} \label{app:covariance}
We start from 
\begin{equation}\label{eq: formula of s}
s =\sum\limits_{k \in \mathcal{K}_{\mathrm{L}}} \frac{\widehat{w}[k]}{\widehat{L}[k] }e_k  + \sum\limits_{n=1}^{N_0} \frac{\widehat{w}[k_n]}{\gamma_0}e_{k_n}.
\end{equation}
Our goal is to compute $r_s(t,\tau) = \mathbb{E}[s(t)s(\tau)]$. We do so by replacing $s(t)$ and $s(\tau)$ with~\eqref{eq: formula of s}. We develop the product and use the relations 
$\mathbb{E} [\widehat{w}[k] \widehat{w}[\ell] ] = \mathbb{E}[ \widehat{w}[k]^2 ] = 0$, $\mathbb{E}[\lvert\widehat{w} [k] \rvert^2] = 1$ for every $k,  \ell \in \mathbb{Z}$, $k\neq \ell$ to deduce that
\begin{equation}
r_s(t,\tau) =\left( \sum_{k\in\mathcal{K}_{\mathrm{L}}} \frac{e_k(t) e_{-k}(\tau)}{\lvert \widehat{L}[k]\rvert^2}  + \frac{1}{\gamma_0^2}\sum_{n=1}^{N_0} e_{k_n}(t) e_{-k_n} (\tau) \right).
\end{equation}
Since $e_k(t) e_{-k}(\tau)=e_k(t-\tau)$, we have shown that ${r_s(t,\tau) = h_{\gamma}(t-\tau)}$, as expected. Then, we obtain~\eqref{eq: esperance GB} by injecting~\eqref{eq: covariance GB} into~\eqref{eq: esp Gaussian process}. Finally, we obtain~\eqref{eq:variancesk} by particularizing~\eqref{eq: esperance GB} with $\nu_m = e_k$. 

\subsection{Proof of Theorem~\ref{th: MMSE}}\label{app: MMSE}
We fix a time $t_0\in \mathbb{T}$. We first obtain the MMSE estimator for $s(t_0)$ (estimation of $s$ at time $t_0$). (Note that ${s(t_0) = \langle s , \Sha(\cdot - t_0)\rangle}$ is well defined because $\Sha (\cdot - t_0) \in \mathcal{H}_{\mathrm{L}}$ by assumption).

The linear MMSE estimator of $s(t_0)$ based on $\mathbf{y}$ is of the form $\tilde{s}_{t_0}=\sum\limits_{m=1}^M u_m y_m$. Because s and $\epsilon$ are Gaussian, the linear MMSE estimator coincides with the MMSE estimator~\cite{Moon2000mathematical}. The orthogonality principle [Section 3.2]~\cite{Moon2000mathematical} then implies that 
\begin{equation}\label{eq: 1 proof MMSE}
\mathbb{E}[y_m(s(t_0)-\tilde{s}_{t_0})]=0, \ \forall m=1 \ldots M.
\end{equation}
We know from Proposition~\ref{prop:covariancebridge} that $\mathbb{E} [\langle s , f \rangle \langle s , g \rangle ] = \langle  h_{\gamma_0} \ast f,g\rangle$. We use this relation to develop the different terms of~\eqref{eq: 1 proof MMSE}. First, we have that
\begin{small}
\begin{align}\label{eq: 2 proof MMSE}
\mathbb{E}[y_m s(t_0)]&=\mathbb{E}[\langle \nu_m , s \rangle s(t_0)] + \mathbb{E}[\epsilon_m s(t_0)] \nonumber \\
&=\mathbb{E} [\langle  \nu_m , s \rangle \langle s , \Sha(\cdot-t_0) \rangle ]  + \underbrace{\mathbb{E}[\epsilon_m]}_{0}\mathbb{E}[ s(t_0)] \nonumber \\
&=(h_{\gamma_0} \ast \nu_m) (t_0) .
\end{align}
\end{small}
As the estimator is of the form $\tilde{s}_{t_0} = \sum\limits_{m=1}^{M}u_m y_m$ and exploiting that $\epsilon$ and $s$ are independent, we have that
\begin{small}
\begin{align}
	 \mathbb{E}[\langle \nu_m , s \rangle y_k]  &=  \mathbb{E}[\langle \nu_m , s \rangle \langle \nu_k, s \rangle]  +  \mathbb{E}[\langle \nu_m , s \rangle \epsilon_k]  = \langle h_{\gamma_0} \ast \nu_m, \nu_k \rangle \nonumber\\
	  \mathbb{E}[\epsilon_m y_k] &=  \mathbb{E}[\epsilon_m \langle \nu_k, s \rangle] +  \mathbb{E}[\epsilon_m \epsilon_k] = \sigma^2 \delta[m-k]
\end{align}
\end{small}
We have therefore that
\begin{small}
\begin{align}\label{eq: 3 proof MMSE}
 \mathbb{E}[y_m \tilde{s}_{t_0}]&= \mathbb{E}[\langle \nu_m , s \rangle \tilde{s}_{t_0}]+ \mathbb{E}[\epsilon_m \tilde{s}_{t_0}]\nonumber \\
 &=\sum_{k=1}^{M} u_k  \mathbb{E}[\langle \nu_m , s \rangle y_k] + \sum_{k=1}^{M} u_k  \mathbb{E}[\epsilon_m y_k]\nonumber \\
  &=\sum_{k=1}^{M} u_k  \langle  h_{\gamma_0} \ast \nu_m, \nu_k \rangle + u_m \sigma_0^2.
\end{align}
\end{small}
\noindent We remark that $\langle  h_{\gamma_0} \ast \nu_m, \nu_k \rangle=[\mathbf{G}]_{m_1,m_2}$ given in~\eqref{eq:Chgamma}. Injecting~\eqref{eq: 2 proof MMSE} and~\eqref{eq: 3 proof MMSE} into~\eqref{eq: 1 proof MMSE}, we have for $ m=1 \ldots M$ that
$(h_{\gamma_0} \ast \nu_m) (t_0)=\sum_{k=1}^{M} u_k  [\mathbf{G}]_{m_1,m_2} + u_m \sigma_0^2.$ Hence,
$\mathbf{u}
=(\mathbf{G}+ \sigma_0^2 \mathbf{I})^{-1} \mathbf{c},$ where 
$\mathbf{c}
= (h_{\gamma_0} * \bm{\nu}) (t_0)$. As $\tilde{s}_{t_0}=\mathbf{u}^{\mathsf{T}}\mathbf{y}$, we finally have that ${\tilde{s}_{t_0}=\sum\limits_{m=1}^M d_m(h_{\gamma_0}\ast \nu_m)}(t_0)$, where ${\mathbf{d}=(d_1, \dots ,d_M)=(\mathbf{G}+ \sigma_0^2 \mathbf{I})^{-1}\mathbf{y}}$.

We have now obtained the form of the MMSE estimator $\tilde{s}_{t_0}$ for $s(t_0)$ at a fixed time $t_0$. We then deduce the MMSE estimator of the complete continuous random process $s: \mathbb{T} \rightarrow \mathbb{R}$ that minimizes $\mathbb{E}[\| s-\tilde{s}\|_{L_2}^2]$ among all the estimators $\tilde{s}$ based on $\mathbf{y}$. We fix an estimator $\tilde{s}$. We have that
\begin{small}
\begin{align}
\mathbb{E}[\| s - \tilde{s}\|_{L_2}^2]&=\mathbb{E}[\int_0^1 (s(t)- \tilde{s}(t))^2 \mathrm{d}t] =\int_0^1 \mathbb{E}[(s(t)-\tilde{s}(t))^2]\mathrm{d}t \nonumber\\
&\geq \int_0^1 \mathbb{E}[(s(t)-\tilde{s}_{t})^2 ]\mathrm{d}t =\mathbb{E}[\| s-\tilde{s}_{\mathrm{MMSE}}\|_{L_2}^2].
\end{align}
\end{small}
\noindent Hence, the function $\tilde{s}_{\mathrm{MMSE}}: t \rightarrow \tilde{s}_t$ is the MMSE estimator of the complete process $s(t)$.

\subsection{Proof of Proposition~\ref{prop: opt pb}} \label{sec:proofRTgamma}
The proof is obtained by following the arguments of Theorem~\ref{th: periodic representer thm} (for existence, unicity, and form of the solution) and Proposition~\ref{prop: interpolation pb} (for the explicit formula of the coefficients ${d}_m$ in~\eqref{eq:optigamma}) with the following simplifications:

\noindent First, the existence and unicity of a solution is now direct. Indeed, the functional to minimize is ${\lVert \mathbf{y} - \mathbf{\nu}(f) \rVert_2^2 + \lambda \lVert f \rVert_{\mathcal{H}_{\mathrm{L}}}^2}$. It is clearly coercive and strictly convex because $\lVert \cdot \rVert_{\mathcal{H}_{\mathrm{L}}}$ is.
Second, the abstract representer theorem can now be applied directly to the Hilbert space $\mathcal{H}_{\mathrm{L}}$. The form of the solution is then directly deduced.
Third, the coefficients $d_m$ are found with the arguments of Appendix~\ref{app: interpolation pb}, except that there is no term for the null-space component (coefficients $b_n$) in that case, hence the system matrix is simpler. 

\subsection{Proof of Proposition~\ref{prop:fouriersampling}} \label{app:Fouriersampling}
We know the expression of $\tilde{s}_{\lambda}$ from Proposition~\ref{prop: opt pb}.
For Fourier sampling, the $\varphi_m$ are complex exponential themselves, given by ${\varphi_m = h* e_{k_m} = \widehat{h}[k_m] e_{k_m}}$, 
while the Gram matrix $\mathbf{G}$ is diagonal since ${\mathbf{G}_{m_1,m_2} = \langle h * e_{k_{m_1}}, e_{k_{m_2}} \rangle =  \widehat{h}[k_{m_1}]  \delta[k_{m_1} - k_{m_2}]}$. Hence, ~\eqref{eq:optigamma} gives that
\begin{equation}
	\tilde{s}_{\lambda} = \sum_{m=1}^M \frac{(\widehat{s}[k_m] + \epsilon_m ) \widehat{h}[k_m]}{\widehat{h}[k_m]+ \lambda} e_{k_m}.
\end{equation}
After simplification, we have that
\begin{equation}
	s - \tilde{s}_\lambda =  \sum_{m=1}^M \left(\frac{\lambda \widehat{s}[k_m]  }{\widehat{h}[k_m]+ \lambda} - \frac{\widehat{h}[k_m] \epsilon_m}{\widehat{h}[k_m]+ \lambda}\right)  e_{k_m} + \sum_{k\notin \mathcal{N}_{\bm{\nu}}} \widehat{s}[k] e_k.
\end{equation}
Exploiting the Fourier-domain independence, we deduce that
\begin{align}
	\mathbb{E}\left[ \lVert s - \tilde{s}_\lambda \rVert_{L_2}^2 \right] 
	=
	&\sum_{m=1}^M \frac{\lambda^2}{( \widehat{h}[k_m] + \lambda)^2} \mathbb{E} \left[ \lvert \widehat{s}[k_m] \rvert^2 \right] \nonumber \\
&	\quad \ + \frac{\widehat{h}[k_m]^2}{( \widehat{h}[k_m] + \lambda)^2} \mathbb{E} \left[ \lvert \epsilon_m \rvert^2 \right]  \nonumber \\
	& +  \sum_{k\notin \mathcal{N}_{\bm{\nu}}} \mathbb{E} \left[ \lvert \widehat{s}[k] \rvert^2 \right].
\end{align} 
From the relations $\mathbb{E} \left[ \lvert \widehat{s}[k] \rvert^2 \right] = \widehat{h}[k]$ (see~\eqref{eq:variancesk}) and $\mathbb{E} \left[ \lvert \epsilon_m \rvert^2 \right] = \sigma_0^2$, we finally obtain~\eqref{eq:FouriersamplingNMSE}.


\bibliographystyle{IEEEtran}
\bibliography{references}

\begin{thebibliography}{10}
\providecommand{\url}[1]{#1}
\csname url@samestyle\endcsname
\providecommand{\newblock}{\relax}
\providecommand{\bibinfo}[2]{#2}
\providecommand{\BIBentrySTDinterwordspacing}{\spaceskip=0pt\relax}
\providecommand{\BIBentryALTinterwordstretchfactor}{4}
\providecommand{\BIBentryALTinterwordspacing}{\spaceskip=\fontdimen2\font plus
\BIBentryALTinterwordstretchfactor\fontdimen3\font minus
  \fontdimen4\font\relax}
\providecommand{\BIBforeignlanguage}[2]{{%
\expandafter\ifx\csname l@#1\endcsname\relax
\typeout{** WARNING: IEEEtran.bst: No hyphenation pattern has been}%
\typeout{** loaded for the language `#1'. Using the pattern for}%
\typeout{** the default language instead.}%
\else
\language=\csname l@#1\endcsname
\fi
#2}}
\providecommand{\BIBdecl}{\relax}
\BIBdecl

\bibitem{Banham1997digital}
M.~Banham and A.~Katsaggelos, ``Digital image restoration,'' \emph{IEEE Signal
  Processing Magazine}, vol.~14, no.~2, pp. 24--41, March 1997.

\bibitem{Karayiannis1990regularization}
N.~Karayiannis and A.~Venetsanopoulos, ``Regularization theory in image
  restoration---{T}he stabilizing functional approach,'' \emph{IEEE
  Transactions on Acoustics, Speech, and Signal Processing}, vol.~38, no.~7,
  pp. 1155--1179, July 1990.

\bibitem{Figueiredo2003algorithm}
M.~Figueiredo and R.~Nowak, ``An {EM} algorithm for wavelet-based image
  restoration,'' \emph{IEEE Transactions on Image Processing}, vol.~12, no.~8,
  pp. 906--916, August 2003.

\bibitem{Afonso2011augmented}
M.~Afonso, J.~Bioucas-Dias, and M.~Figueiredo, ``An augmented {L}agrangian
  approach to the constrained optimization formulation of imaging inverse
  problems,'' \emph{IEEE Transactions on Image Processing}, vol.~20, no.~3, pp.
  681--695, March 2011.

\bibitem{Bertero1998introduction}
M.~Bertero and P.~Boccacci, \emph{Introduction to {I}nverse {P}roblems in
  {I}maging}.\hskip 1em plus 0.5em minus 0.4em\relax CRC press, 1998.

\bibitem{Adcock2015generalized}
B.~Adcock and A.~Hansen, ``Generalized sampling and infinite-dimensional
  compressed sensing,'' \emph{Foundations of Computational Mathematics},
  vol.~16, no.~5, pp. 1263--1323, October 2016.

\bibitem{Papoulis1977generalized}
A.~Papoulis, ``Generalized sampling expansion,'' \emph{IEEE Transactions on
  Circuits and Systems}, vol.~24, no.~11, pp. 652--654, November 1977.

\bibitem{Eldar2006minimum}
Y.~Eldar and T.~Dvorkind, ``A minimum squared-error framework for generalized
  sampling,'' \emph{IEEE Transactions on Signal Processing}, vol.~54, no.~6,
  pp. 2155--2167, June 2006.

\bibitem{Piccolomini2002regularization}
E.~Piccolomini, F.~Zama, G.~Zanghirati, and A.~Formiconi, ``Regularization
  methods in dynamic {MRI},'' \emph{Applied Mathematics and Computation}, vol.
  132, no.~2, pp. 325--339, November 2002.

\bibitem{Bostan2013Sparse}
E.~Bostan, U.~Kamilov, M.~Nilchian, and M.~Unser, ``Sparse stochastic processes
  and discretization of linear inverse problems,'' \emph{{IEEE} Transactions on
  Image Processing}, vol.~22, no.~7, pp. 2699--2710, July 2013.

\bibitem{Tikhonov1963solution}
A.~Tikhonov, ``Solution of incorrectly formulated problems and the
  regularization method,'' \emph{Soviet Mathematics Doklady}, vol.~4, pp.
  1035--1038, 1963.

\bibitem{cassel2013variational}
K.~Cassel, \emph{Variational {M}ethods with {A}pplications in {S}cience and
  {E}ngineering}.\hskip 1em plus 0.5em minus 0.4em\relax Cambridge University
  Press, 2013.

\bibitem{Moon2000mathematical}
T.~Moon and W.~Stirling, \emph{Mathematical {M}ethods and {A}lgorithms for
  {S}ignal {P}rocessing}.\hskip 1em plus 0.5em minus 0.4em\relax Prentice Hall
  Upper Saddle River, NJ, 2000, vol.~1.

\bibitem{Cohen1994part}
F.~S. Cohen and J.-Y. Wang, ``Part {I}: Modeling image curves using invariant
  3-{D} object curve models---{A} path to 3-{D} recognition and shape
  estimation from image contours,'' \emph{IEEE Transactions on Pattern Analysis
  and Machine Intelligence}, vol.~16, no.~1, pp. 1--12, January 1994.

\bibitem{Delgado2012ellipse}
R.~Delgado-Gonzalo, P.~Th{\'{e}}venaz, C.~Seelamantula, and M.~Unser, ``Snakes
  with an ellipse-reproducing property,'' \emph{{IEEE} Transactions on Image
  Processing}, vol.~21, no.~3, pp. 1258--1271, March 2012.

\bibitem{Badoual2017subdivision}
A.~Badoual, D.~Schmitter, V.~Uhlmann, and M.~Unser, ``Multiresolution
  subdivision snakes,'' \emph{{IEEE} Transactions on Image Processing},
  vol.~26, no.~3, pp. 1188--1201, March 2017.

\bibitem{Vetterli2002FRI}
M.~Vetterli, P.~Marziliano, and T.~Blu, ``Sampling signals with finite rate of
  innovation,'' \emph{{IEEE} Transactions on Signal Processing}, vol.~50,
  no.~6, pp. 1417--1428, June 2002.

\bibitem{Maravic2005sampling}
I.~Maravic and M.~Vetterli, ``Sampling and reconstruction of signals with
  finite rate of innovation in the presence of noise,'' \emph{IEEE Transactions
  on Signal Processing}, vol.~53, no.~8, pp. 2788--2805, August 2005.

\bibitem{blu2008sparse}
T.~Blu, P.-L. Dragotti, M.~Vetterli, P.~Marziliano, and L.~Coulot, ``Sparse
  sampling of signal innovations,'' \emph{IEEE Signal Processing Magazine},
  vol.~25, no.~2, pp. 31--40, March 2008.

\bibitem{Jacob2002sampling}
M.~Jacob, T.~Blu, and M.~Unser, ``Sampling of periodic signals: {A}
  quantitative error analysis,'' \emph{{IEEE} Transactions on Signal
  Processing}, vol.~50, no.~5, pp. 1153--1159, May 2002.

\bibitem{Triebel2008function}
H.~Triebel, \emph{Function {S}paces and {W}avelets on {D}omains}, ser. EMS
  Tracts in Mathematics.\hskip 1em plus 0.5em minus 0.4em\relax European
  Mathematical Society (EMS), Z\"urich, 2008, vol.~7.

\bibitem{Fageot2017besov}
J.~Fageot, M.~Unser, and J.~P. Ward, ``On the {B}esov regularity of periodic
  {L}{\'e}vy noises,'' \emph{Applied and Computational Harmonic Analysis},
  vol.~42, no.~1, pp. 21--36, January 2017.

\bibitem{Badoual2016inner}
A.~Badoual, D.~Schmitter, and M.~Unser, ``An inner-product calculus for
  periodic functions and curves,'' \emph{{IEEE} Signal Processing Letters},
  vol.~23, no.~6, pp. 878--882, June 2016.

\bibitem{Scholkopf2001generalized}
B.~Sch{\"o}lkopf, R.~Herbrich, and A.~Smola, ``A generalized representer
  theorem,'' in \emph{Computational Learning Theory}.\hskip 1em plus 0.5em
  minus 0.4em\relax Springer, 2001, pp. 416--426.

\bibitem{Schoenberg1964spline}
I.~Schoenberg, ``Spline functions and the problem of graduation,''
  \emph{Proceedings of the National Academy of Sciences}, vol.~52, no.~4, pp.
  947--950, August 1964.

\bibitem{Unser2007self}
M.~Unser and T.~Blu, ``Self-similarity: {P}art {I}---{S}plines and operators,''
  \emph{{IEEE} Transactions on Signal Processing}, vol.~55, no.~4, pp.
  1352--1363, April 2007.

\bibitem{Uhlmann2016hermite}
V.~Uhlmann, J.~Fageot, and M.~Unser, ``Hermite snakes with control of
  tangents,'' \emph{IEEE Transactions on Image Processing}, vol.~25, no.~6, pp.
  2803--2816, June 2016.

\bibitem{Condat2011quantitative}
L.~Condat and T.~M\"oller, ``Quantitative error analysis for the reconstruction
  of derivatives,'' \emph{IEEE Transactions on Signal Processing}, vol.~59,
  no.~6, pp. 2965--2969, June 2011.

\bibitem{Candes2006sparse}
E.~J. Cand{\`e}s, J.~Romberg, and T.~Tao, ``Robust uncertainty principles:
  Exact signal reconstruction from highly incomplete frequency information,''
  \emph{IEEE Transactions on Information Theory}, vol.~52, no.~2, pp. 489--509,
  February 2006.

\bibitem{Donoho2006}
D.~Donoho, ``Compressed sensing,'' \emph{IEEE Transactions on Information
  Theory}, vol.~52, no.~4, pp. 1289--1306, April 2006.

\bibitem{Denoyelle2015support}
Q.~Denoyelle, V.~Duval, and G.~Peyr{\'e}, ``Support recovery for sparse
  deconvolution of positive measures,'' \emph{arXiv preprint arXiv:1506.08264},
  2015.

\bibitem{Unser2016splines}
M.~Unser, J.~Fageot, and J.~P. Ward, ``Splines are universal solutions of
  linear inverse problems with generalized {TV} regularization,'' \emph{{SIAM}
  Review}, vol.~59, no.~4, pp. 769--793, December 2017.

\bibitem{gupta2018continuous}
H.~Gupta, J.~Fageot, and M.~Unser, ``Continuous-domain solutions of linear
  inverse problems with {T}ikhonov \textit{vs.} generalized {TV}
  regularization,'' \emph{arXiv preprint arXiv:1802.01344, to appear in IEEE
  Transactions on Signal Processing}, 2018.

\bibitem{Tarantola2005inverse}
A.~Tarantola, \emph{Inverse Problem Theory and Methods for Model Parameter
  Estimation}.\hskip 1em plus 0.5em minus 0.4em\relax SIAM, 2005.

\bibitem{Berlinet2011reproducing}
A.~Berlinet and C.~Thomas-Agnan, \emph{Reproducing Kernel Hilbert Spaces in
  Probability and Statistics}.\hskip 1em plus 0.5em minus 0.4em\relax Springer
  Science \& Business Media, 2011.

\bibitem{Kailath1968innovationsA}
T.~Kailath, ``An innovations approach to least-squares estimation--{P}art {I}:
  Linear filtering in additive white noise,'' \emph{IEEE Transactions on
  Automatic Control}, vol.~13, no.~6, pp. 646--655, December 1968.

\bibitem{Unser2014unifiedContinuous}
M.~Unser, P.~D. Tafti, and Q.~Sun, ``A unified formulation of {G}aussian versus
  sparse stochastic processes---{P}art {I}: {C}ontinuous-domain theory,''
  \emph{{IEEE} Transactions on Information Theory}, vol.~60, no.~3, pp.
  1945--1962, March 2014.

\bibitem{Kimeldorf1970spline}
G.~Kimeldorf and G.~Wahba, ``Spline functions and stochastic processes,''
  \emph{Sankhy{\=a}: The Indian Journal of Statistics, Series A}, vol.~32,
  no.~2, pp. 173--180, June 1970.

\bibitem{Uhlmann2015SampTA}
V.~Uhlmann, J.~Fageot, H.~Gupta, and M.~Unser, ``Statistical optimality of
  {H}ermite splines,'' in \emph{Proceedings of the Eleventh International
  Workshop on Sampling Theory and Applications ({SampTA'15})}, Washington DC,
  USA, May 25-29, 2015, pp. 226--230.

\bibitem{Blu2007self}
T.~Blu and M.~Unser, ``Self-similarity: {P}art {II}---{O}ptimal estimation of
  fractal processes,'' \emph{{IEEE} Transactions on Signal Processing},
  vol.~55, no.~4, pp. 1364--1378, April 2007.

\bibitem{Unser2014sparse}
M.~Unser and P.~D. Tafti, \emph{An Introduction to Sparse Stochastic
  Processes}.\hskip 1em plus 0.5em minus 0.4em\relax Cambridge University
  Press, 2014.

\bibitem{Wahba1990spline}
G.~Wahba, \emph{Spline Models for Observational Data}.\hskip 1em plus 0.5em
  minus 0.4em\relax SIAM, 1990.

\bibitem{Unser2005generalized}
M.~Unser and T.~Blu, ``Generalized smoothing splines and the optimal
  discretization of the {W}iener filter,'' \emph{IEEE Transactions on Signal
  Processing}, vol.~53, no.~6, pp. 2146--2159, June 2005.

\bibitem{Schoenberg1964trigonometric}
I.~Schoenberg, ``On trigonometric spline interpolation,'' \emph{Journal of
  Mathematics and Mechanics}, vol.~13, no.~5, pp. 795--825, 1964.

\bibitem{Golomb1968approximation}
M.~Golomb, ``Approximation by periodic spline interpolants on uniform meshes,''
  \emph{Journal of Approximation Theory}, vol.~1, no.~1, pp. 26--65, June 1968.

\bibitem{de1978practical}
C.~D. Boor, \emph{A Practical Guide to Splines}.\hskip 1em plus 0.5em minus
  0.4em\relax Springer-Verlag New York, 1978, vol.~27.

\bibitem{Schwartz1966distributions}
L.~Schwartz, \emph{Th\'eorie des distributions}.\hskip 1em plus 0.5em minus
  0.4em\relax Hermann, 1966.

\bibitem{Schoenberg1973cardinal}
I.~Schoenberg, \emph{Cardinal {S}pline {I}nterpolation}.\hskip 1em plus 0.5em
  minus 0.4em\relax Philadelphia, PA: SIAM, 1973.

\bibitem{Unser1999splines}
M.~Unser, ``Splines: A perfect fit for signal and image processing,''
  \emph{IEEE Signal Processing Magazine}, vol.~16, no.~6, pp. 22--38, November
  1999.

\bibitem{schumaker2007spline}
L.~Schumaker, \emph{Spline Functions: Basic Theory}.\hskip 1em plus 0.5em minus
  0.4em\relax Cambridge University Press, 2007.

\bibitem{Schultz1967Lsplines}
M.~Schultz and R.~Varga, ``L-splines,'' \emph{Numerische Mathematik}, vol.~10,
  no.~4, pp. 345--369, November 1967.

\bibitem{Unser2005cardinal}
M.~Unser and T.~Blu, ``Cardinal exponential splines: {P}art {I}---{T}heory and
  filtering algorithms,'' \emph{{IEEE} Transactions on Signal Processing},
  vol.~53, no.~4, pp. 1425--1438, April 2005.

\bibitem{Unser2005think}
M.~Unser, ``Cardinal exponential splines: {P}art {II}---{T}hink analog, act
  digital,'' \emph{{IEEE} Transactions on Signal Processing}, vol.~53, no.~4,
  pp. 1439--1449, April 2005.

\bibitem{Panda2006fractional}
R.~Panda and M.~Dash, ``Fractional generalized splines and signal processing,''
  \emph{Signal Processing}, vol.~86, no.~9, pp. 2340--2350, September 2006.

\bibitem{Unser2016representer}
M.~Unser, J.~Fageot, and H.~Gupta, ``Representer theorems for
  sparsity-promoting $\ell_{1}$ regularization,'' \emph{{IEEE} Transactions on
  Information Theory}, vol.~62, no.~9, pp. 5167--5180, September 2016.

\bibitem{Goodman1963}
N.~R. Goodman, ``Statistical analysis based on a certain multivariate complex
  {G}aussian distribution (an introduction),'' \emph{The Annals of Mathematical
  Statistics}, vol.~34, no.~1, pp. 152--177, March 1963.

\bibitem{Revuz2013continuous}
D.~Revuz and M.~Yor, \emph{Continuous Martingales and Brownian Motion}.\hskip
  1em plus 0.5em minus 0.4em\relax Springer Science \& Business Media, 2013,
  vol. 293.

\bibitem{Simon2003distributions}
B.~Simon, ``Distributions and their {H}ermite expansions,'' \emph{Journal of
  Mathematical Physics}, vol.~12, no.~1, pp. 140--148, October 2003.

\bibitem{Karatzas2012brownian}
I.~Karatzas and S.~Shreve, \emph{Brownian Motion and Stochastic
  Calculus}.\hskip 1em plus 0.5em minus 0.4em\relax Springer Science \&
  Business Media, 2012, vol. 113.

\bibitem{Elad10}
M.~Elad, \emph{Sparse and Redundant Representations: From Theory to
  Applications in Signal and Image Processing}.\hskip 1em plus 0.5em minus
  0.4em\relax Springer, 2010.

\bibitem{mallat2008wavelet}
S.~Mallat, \emph{A Wavelet Tour of Signal Processing: The Sparse Way}.\hskip
  1em plus 0.5em minus 0.4em\relax Academic press, 2008.

\bibitem{candes2014towards}
E.~Cand{\`e}s and C.~Fernandez-Granda, ``Towards a mathematical theory of
  super-resolution,'' \emph{Communications on Pure and Applied Mathematics},
  vol.~67, no.~6, pp. 906--956, June 2014.

\bibitem{Wendland2004scattered}
H.~Wendland, \emph{Scattered Data Approximation}.\hskip 1em plus 0.5em minus
  0.4em\relax Cambridge university press, 2004, vol.~17.

\end{thebibliography}

\end{document}